\documentclass[12pt,a4paper]{article}
\usepackage{a4wide,amsfonts,amsmath,latexsym,amssymb,euscript,eufrak,graphicx,units,mathrsfs}
\usepackage[utf8]{inputenc}
\usepackage{amsmath}
\usepackage{amsfonts}
\usepackage{amssymb}
\usepackage{amsthm}
\usepackage{floatrow}
\usepackage{blindtext}

\usepackage[usenames,dvipsnames]{color}

\numberwithin{equation}{section}

\newcommand{\R}{\mathbb{R}}

\newtheorem{theorem}{Theorem}[section]

\newtheorem{lemma}{Lemma}[section]
\newtheorem{remark}{Remark}[section]
\newtheorem{proposition}{Proposition}[section]

\begin{document}

\renewcommand{\thefootnote}{\arabic{footnote}}

\begin{center}
{\Large \textbf{Estimating drift parameters in a non-ergodic Gaussian Vasicek-type model }} \\[0pt]
~\\[0pt]
Khalifa Es-Sebaiy\footnote{ Department of Mathematics, Faculty of
Science, Kuwait University, Kuwait. E-mail:
 \texttt{khalifa.essebaiy@ku.edu.kw}}\
  Mohammed Es.Sebaiy\footnote{National School of Applied Sciences-Marrakech, Cadi Ayyad University, Marrakech, Moroccco. E-mail:
 \texttt{mohammedsebaiy@gmail.com}}
\\[0pt]
~\\[0pt]
 Kuwait University and Cadi Ayyad University\\[0pt]
~\\[0pt]
\end{center}

\begin{abstract}
We study a problem of parameter estimation for a non-ergodic
Gaussian Vasicek-type model defined as $dX_t=\theta(\mu+
X_t)dt+dG_t,\ t\geq0$ with  unknown parameters $\theta>0$,
$\mu\in\R$ and $\alpha:=\theta\mu$, where $G$ is a Gaussian process.
We provide least square-type estimators
$(\widetilde{\theta}_T,\widetilde{\mu}_T)$ and
$(\widetilde{\theta}_T,\widetilde{\alpha}_T)$, respectively,
  for   $(\theta,\mu)$ and $(\theta,\alpha)$ based a continuous-time observation of $\{X_t,\ t\in[0,T]\}$ as
$T\rightarrow\infty$.
 Our aim is to derive some sufficient conditions on the driving Gaussian process
$G$ in order  to ensure the strongly consistency and the joint
asymptotic distribution of
$(\widetilde{\theta}_T,\widetilde{\mu}_T)$ and
$(\widetilde{\theta}_T,\widetilde{\alpha}_T)$. Moreover, we obtain
that the limit distribution of $\widetilde{\theta}_T$ is a
Cauchy-type distribution, and $\widetilde{\mu}_T$ and
$\widetilde{\alpha}_T$ are asymptotically normal.
 We apply our result to fractional Vasicek,  subfractional Vasicek and bifractional Vasicek
 processes. This work extends the results of \cite{EEO}
studied in  the case where $\mu=0$.
\end{abstract}


\noindent {\bf 2010 AMS Classification Numbers:}  60G15; 60G22;
62F12; 62M09; 62M86.

\section{Introduction}
Let   $G:=\{G_t,t\geq0\}$ be a centered   Gaussian process
satisfying the following assumption
\begin{itemize}
\item[$(\mathcal{A}_1)$]There exist   constants  $c>0$ and $ \gamma\in(0,1)$ such
that
 for every $s,t\geq0$, \[G_0=0, \qquad E\left[\left(G_t-G_s\right)^2\right]\leq
 c\,\left|t-s\right|^{2\gamma}.\]
\end{itemize}Note that, if  $(\mathcal{A}_1)$ holds, then by  the
Kolmogorov-Centsov theorem, we can conclude that for all $\varepsilon\in\left(0,\gamma\right)$, the process $G$
 admits a modification with $(\gamma-\varepsilon)-$H\"older
continuous paths, still denoted $G$ in the sequel.

In the present paper, our goal is to estimate jointly the drift
parameters of the Gaussian Vasicek-type (also called mean-reverting
Ornstein-Uhlenbeck) process
 $X:=\{X_t,t\geq0\}$ that is
  defined as the unique (pathwise) solution to
\begin{equation}
  \label{GV}
  X_0=0,\qquad dX_t=\theta\left(\mu+X_t\right)dt+dG_t,\quad
t\geq0,
\end{equation}
  where $\theta>0$  and  $\mu\in\R$ are considered as
  unknown parameters. When $G$ is a standard Brownian
motion, the model (\ref{GV}) with $\mu = 0$ was originally proposed
by Ornstein and Uhlenbeck and then it was generalized by Vasicek,
see \cite{vasicek}.

In recent years, several researchers have been interested in
studying statistical estimation problems for Gaussian
Ornstein-Uhlenbeck processes. Estimation of the drift parameters  in
fractional-noise-driven Ornstein-Uhlenbeck processes  is a problem
that is both well-motivated by practical needs and theoretically
challenging. In the finance context, our practical motivation to
study this estimation problem is to provide tools to understand
volatility modeling in finance. Indeed, any mean-reverting model in
discrete or continuous time can be taken as a model for stochastic
volatility. Let us mention some  important results in this field
where  the volatility exhibits long-memory, which means that the
volatility today is correlated to past volatility values with a
dependence that decays very slowly. Following the approach of
\cite{CCR,CR98}, the authors of \cite{CV12a,CV12b} considered the
problem of option pricing under a stochastic volatility model that
exhibits long-range dependence. More precisely they assumed that the
dynamics of the volatility are described by the equation (\ref{GV}),
where the driving process $G$ is a standard fractional Brownian
motion (fBm) with Hurst parameter $H$ is greater than $1/2$. On the
other hand, the paper \cite{GJR} on rough volatility contends that
the short-time behavior indicates that the Hurst parameter $H$ in
the volatility is less than $1/2$.

 An example of interesting problem related to (\ref{GV})  is the
statistical estimation of $\mu$ and $\theta$ when one observes the
whole trajectory of $X$. In order to estimate the unknown parameters
 $\theta$ and $\mu $ when the whole trajectory of $X$ defined in
(\ref{GV}) is observed, we will first consider the classical least
squares estimators (LSEs) $\widehat{\theta}_T$ and
$\widehat{\alpha}_T$ for $\theta$ and $\alpha:=\mu\theta $,
respectively. By   minimizing (formally) the function
\[F(\theta,\alpha) = \int_0^T\left|\dot{{X}}_s-\left(\alpha+\theta X_s\right)\right|^2ds,\]
we obtain
\begin{equation}\label{est1-theta}
\widehat{\theta}_T=\frac{T\int_0^T X_s d X_s-X_T\int_0^T X_s d
s}{T\int_0^T X_s^2ds-\left(\int_0^T X_s d s\right)^2}
\end{equation}
and
\begin{equation}\label{est1-alpha}
\widehat{\alpha}_T=\frac{X_T\int_0^T X_s^2 d s-\int_0^T X_s d
X_s\int_0^T X_sds}{T\int_0^T X_s^2ds-\left(\int_0^T X_s d
s\right)^2}.
\end{equation}
More precisely, $(\widehat{\theta}_T,\widehat{\alpha}_T)$ is the
solution of the system $\frac{\partial
F}{\partial\theta}(\theta,\alpha)=0$, $\frac{\partial
F}{\partial\alpha}(\theta,\alpha)=0$. Moreover, the expressions
given in (\ref{est1-theta}) and (\ref{est1-alpha}) are well-defined
for $\frac12<\gamma<1$, since the stochastic integral $\int_0^T X_s
d X_s$ is well-defined in the Young sense (see Appendix) for
$\frac12<\gamma<1$ only, by using $(\mathcal{A}_1)$, (\ref{GV-zeta})
and Remark \ref{remark-X}. Further, thanks to (\ref{IBP}), we can
write $\int_0^T X_s d X_s=\frac12 X_T^2$ for $\frac12<\gamma<1$.
Thus, we can extend the estimators $\widehat{\theta}_T$ and
$\widehat{\alpha}_T$ to all $0<\gamma<1$, as follows
\begin{equation}\label{est3-theta}
\widetilde{\theta}_T=\frac{\frac12 TX_T^2 -X_T\int_0^T X_s d
s}{T\int_0^T X_s^2ds-\left(\int_0^T X_s d s\right)^2}
\end{equation}
and
\begin{equation}\label{est3-alpha}
\widetilde{\alpha}_T=\frac{X_T\int_0^T X_s^2 d s-\frac12
X_T^2\int_0^T X_sds}{T\int_0^T X_s^2ds-\left(\int_0^T X_s d
s\right)^2}.
\end{equation}
Furthermore, we can  obtain a least squares-type estimator
$\widetilde{\mu}_T$ for $\mu$, that is the statistic
\begin{equation}\label{est3-mu}
\widetilde{\mu}_T=\widetilde{\alpha}_T/\widetilde{\theta}_T=\frac{\int_0^T
X_s^2 d s-\frac12 X_T\int_0^T X_sds}{\frac12 TX_T -\int_0^T X_s d
s}.
\end{equation}
Now, notice that the estimators (\ref{est3-theta}),
(\ref{est3-alpha}) and (\ref{est3-mu}) are well-defined for all
$\gamma\in(0,1)$, and not only for $\gamma\in(\frac12,1)$. This then
allows us to consider $(\widetilde{\theta}_T,\widetilde{\mu}_T)$ as
estimator to estimate the parameters $(\theta,\mu)$  of the equation
(\ref{GV}), and $(\widetilde{\theta}_T,\widetilde{\alpha}_T)$ as
estimator to estimate the parameters $(\theta,\alpha)$  of the
process (\ref{GV}) in the form
\begin{equation}X_0=0,\quad dX_t=\left(\alpha+\theta X_t\right)dt+dG_t,\quad
t\geq0,\label{GV-alpha}
\end{equation}
for all $\gamma\in(0,1)$.\\

We apply our approach to some Vasicek Gaussian processes as follows:

  \noindent \emph{\underline{Fractional  Vasicek process}}:\\
Suppose that the process $G$  given in (\ref{GV})  is a fractional
Brownian motion with Hurst parameter $H\in(0,1)$. When
$H\in(\frac12,1)$, the parameter estimation for  $\theta$ and $\mu$
has been studied in \cite{XY} by using the LSEs $\widehat{\theta}_T$
and $\widehat{\mu}_T$ which coincide, respectively, with
$\widetilde{\theta}_T$ and $\widetilde{\mu}_T$ for
$H\in(\frac12,1)$. Here we present a study valid for all
$H\in(0,1)$. Moreover, we study the joint asymptotic distribution of
$(\widetilde{\theta}_T,\widetilde{\mu}_T)$
(see Section \ref{section-fBmV}).\\

 \noindent \emph{\underline{Subfractional  Vasicek process}}:\\
Assume that the process $G$  given in (\ref{GV})  is a subfractional
Brownian motion $S^H$ with  parameter $H\in(0,1)$, that is, $S^H$ is
a centered Gaussian process with covariance function
\[E\left(S^H_tS^H_s\right)=t^{2H}+s^{2H}-\frac{1}{2}\left((t+s)^{2H}+|t-s|^{2H}\right);\quad s,\ t\geq0.\] For
$H>\frac12$, using the LSEs $\widehat{\theta}_T$ and
$\widehat{\mu}_T$ which also coincide, respectively, with
$\widetilde{\theta}_T$ and $\widetilde{\mu}_T$, the statistical
estimation  for  $\theta$ and $\mu$ has been discussed in
\cite{XXY}. But the proof of the asymptotic distribution (3.32) of
$\widehat{\mu}_T$ in \cite{XXY} relies on a possible  flawed
technique because as $T\rightarrow\infty$, the value of the limit
given in (3.32) depends on $T$, which gives a contradiction. Here,
we give a solution for this problem, and we extend the results of
\cite{XXY} to all $H\in(0,1)$, and  we also study the joint
asymptotic distribution of
$(\widetilde{\theta}_T,\widetilde{\mu}_T)$ (see Section \ref{section-subfBmV}).\\

\noindent \emph{\underline{Bifractional  Vasicek process}}:\\
 To the best of our knowledge there is no study of the  problem of estimating the drift of (\ref{GV}) in the
case when $G$ is a bifractional Brownian motion  $B^{H,K}$ with
parameters $(H,K)\in (0,1)^2$, that is,   $B^{H,K}$ is a centered
Gaussian process with the covariance function
\begin{eqnarray*}
E(B^{H,K}_sB^{H,K}_t)=\frac{1}{2^K}\left(\left(t^{2H}+s^{2H}\right)^K-|t-s|^{2HK}\right);
\quad s,t\geq0.
\end{eqnarray*}Section \ref{section-bifBmV} is devoted to  this question.\\

More recently, the paper \cite{AAE} considered   the least
square-type estimators  (\ref{est3-theta}) and (\ref{est3-mu}) as
estimators of the drift parameters $\theta$ and $\mu$ for the
so-called mean-reverting Ornstein-Uhlenbeck process of the second
kind (also called Vasicek model of the second kind) $\{X_t,t\geq0\}$
defined as $dX_t=\theta(\mu+ X_t)dt+dY_{t,G}^{(1)},\ t\geq0$,  where
$Y_{t,G}^{(1)}:=\int_0^t e^{-s}dG_{a_s}$  with
$a_t=He^{\frac{t}{H}}$,
 and $\{G_{t},t\geq0\}$ is a Gaussian process.

Let us also mention that  similar drift statistical problems for
other Vasicek models were recently studied.  Let us describe what is
known about these parameter estimation problems.
 Let  $B^H:= \left\{B_t^H, t\geq 0\right\}$ denote a fractional Brownian motion (fBm) with Hurst
parameter $H\in(0,1)$. Consider the following fractional
Vasicek-type model driven by $B^H$
\begin{eqnarray}
  dX_{t} &=& \theta\left(\mu+X_t\right)dt+dB_{t}^H,\ \ X_{0}=0,\label{FV-model}
\end{eqnarray}
where $\theta,\  \mu\in \R$ are unknown parameters. Notice that the
process
 (\ref{FV-model}) is   {\it
ergodic} if $\theta <0,\ \mu=0$ and $X_0=\int_{-\infty}^0e^{-\theta
s}dB_{s}^H$. Otherwise,  the process  (\ref{FV-model}) is {\it
non-ergodic} if $\theta>0$.
\\
Now we   recall  several approaches to estimate the parameters of
(\ref{FV-model}). For the maximum likelihood estimation  approach,
in general the techniques used to construct maximum likelihood
estimators (MLEs) for the drift parameters of (\ref{FV-model})
 are based on Girsanov transforms for fBm and depend on  the properties of the deterministic fractional
operators (determined by the Hurst parameter) related to the fBm. In
general, the MLE is not easily computable.   In particular, it
relies on being able to   constitute a discretization of an MLE. For
a more recent comprehensive discussion via this method, we refer to
\cite{KL,TXY}.

  A least squares approach has been also considered by several researchers to study statistical estimation problems for
(\ref{FV-model}). Let us mention some works in this direction: in
the case when $\theta<0$, the statistical estimation   for the
parameters $\mu,\ \theta$ based on continuous-time  observations of
$\{X_t,\ t\in[0,T]\}$ as $T\rightarrow\infty$, has been studied by
several papers, for instance  \cite{DFW,BEV,NT,XY} and the
references therein. When $\mu=0$ in (\ref{FV-model}), the estimation
of $\theta$ has been investigated by using least squares method as
follows: the case of ergodic-type fractional Ornstein-Uhlenbeck
processes, corresponding to $\theta<0$, has been considered in
\cite{HN,EEV,HNZ}, and the case non-ergodic fractional
Ornstein-Uhlenbeck processes has been studied in \cite{EEO,EAA}.

Finally, using a  method of moments, the work \cite{EV} used
Malliavin-calculus advances (see
 \cite{NP-book}) to provide new techniques to statistical inference for stochastic
differential equations related to stationary Gaussian processes, and
its  result has been  used  to study drift parameter estimation
problems for some stochastic differential equations driven by
fractional Brownian motion with fixed-time-step observations, in
particular for the fractional Ornstein-Uhlenbeck
 given in (\ref{FV-model}), where $\mu=0$ and $\theta<0$. Similarly, in
\cite{DEV} the authors studied an  estimator problem for the
parameter $\theta$ in (\ref{FV-model}), where the fractional
Brownian motion is replaced with a general Gaussian
 process.

The rest of the paper is structured as follows. In Section
\ref{sect-auxil} we   analyze some pathwise properties  of the
Vasicek process (\ref{GV}). In Section \ref{sect-asymp} we derive
some sufficient conditions on the driving Gaussian process $G$,
which guarantees the strong consistency and the joint asymptotic
distribution of $(\widetilde{\theta}_T,\widetilde{\mu}_T)$ and
$(\widetilde{\theta}_T,\widetilde{\alpha}_T)$. Section
\ref{sect-appli} is devoted to apply our approach to fractional
Vasicek, subfractional Vasicek and bifractional Vasicek processes.

\section{Notations and auxiliary results}\label{sect-auxil} In this section we   study pathwise properties of
the non-ergodic  Vasicek-type model (\ref{GV}). These properties
will be needed  in order to analyze the asymptotic behavior of the
LSEs $(\widetilde{\theta}_T,\widetilde{\mu}_T)$ and
$(\widetilde{\theta}_T,\widetilde{\alpha}_T)$.
\\
  Because (\ref{GV}) is
linear, it is immediate to solve it explicitly; one then gets the
following formula
\begin{equation}\label{explicit-GV}
X_t=\mu\left(e^{\theta t}-1\right)+e^{\theta t}\int_0^te^{-\theta
s}dG_s,\quad t\geq 0,
\end{equation}
where the integral with respect to $G$ is understood in Young sense (see Appendix).\\
Let us introduce the following processes, for every $t\geq0$,
\begin{equation}\label{zeta}
\zeta_t:=\int_0^te^{-\theta s}dG_s,\quad Z_t:=\int_0^te^{-\theta s}
G_sds,\quad \Sigma_t:=\int_0^tX_sds.
\end{equation}
Thus, using (\ref{explicit-GV}), we can write
\begin{equation}\label{GV-zeta}
X_t= \mu \left(e^{\theta t}-1\right)+e^{\theta t}\zeta_t.
\end{equation}
Furthermore, by (\ref{GV}),
\begin{eqnarray}
X_t=\mu\theta t+\theta \Sigma_t+ G_t.\label{GV-Sigma}
\end{eqnarray}
Moreover, applying  the formula (\ref{IBP}), we have
\begin{eqnarray}\label{zeta-decomposition}
\zeta_t=e^{-\theta t}G_t+ \theta \int_0^te^{-\theta
s}G_sds=e^{-\theta t}G_t+\theta Z_t.
\end{eqnarray}

Here we will discuss some pathwise properties of $\zeta$  and $X$.
\begin{lemma}[\cite{EEO}]\label{lemma-zeta}
Assume that $(\mathcal{A}_1)$ holds with  $\gamma\in(0,1)$. Let
$\zeta $ be given by (\ref{zeta}). Then for all $\varepsilon
\in(0,\gamma)$ the process $\zeta$ admits a modification with
$(\gamma-\varepsilon)$-H\"{o}lder continuous paths, still denoted
$\zeta$ in the sequel.
\\
Moreover,
\begin{eqnarray}Z_T \longrightarrow Z_{\infty}
:= \int_0^{\infty}e^{-\theta s}G_sds,\qquad \zeta_T \longrightarrow
\zeta_\infty :=\theta Z_{\infty}\label{zeta-cv}
\end{eqnarray} almost surely and in $L^2(\Omega)$ as
$T\rightarrow\infty$.
\end{lemma}
\begin{remark}\label{remark-X} Thanks to Lemma \ref{lemma-zeta}, (\ref{GV-zeta})  and (\ref{zeta}) we
deduce that the processes $X$ and $\Sigma$ have H\"{o}der continuous
paths of orders $(\gamma-\varepsilon)$, and $(1-\varepsilon)$ for
any $\varepsilon>0$, respectively.
\end{remark}

We will make use of the following two technical lemmas.
\begin{lemma}Suppose that $(\mathcal{A}_1)$ holds. Then  for any $\gamma<\delta\leq 1$, we have almost
surely, as $T\rightarrow\infty$,
\begin{eqnarray}&&\frac{G_T}{T^{\delta}} \longrightarrow0,\label{cv-G}\\
&&\frac{e^{-\theta T}}{T}\int_0^T|G_tX_t|dt \longrightarrow0.\label{cv-quad-G}\end{eqnarray}
\end{lemma}
\begin{proof}Notice that any constant appearing in our computations below is
understood as a generic constant which might change from line to
line, and depend only on the parameters $\gamma,q,m$. Let us prove
(\ref{cv-G}). Let $\gamma<\delta\leq 1$.
 By Borel-Cantelli's lemma, it is
sufficient to prove that, for any $\varepsilon > 0$,
\begin{eqnarray*}
\sum_{n\geq0} P\left(\sup_{n\leq T \leq
n+1}\left|\frac{G_T}{T^{\delta}}\right|>\varepsilon\right) <\infty.
\end{eqnarray*}
Let $q>0$ such that $q(\delta-\gamma)>1$ and $q\gamma>1$. Applying
Markov's inequality, we obtain
\begin{eqnarray*}
 P\left(\sup_{n\leq T \leq
n+1}\left|\frac{G_T}{T^{\delta}}\right|>\varepsilon\right) &\leq&
\frac{1}{\varepsilon^q}E\left[\sup_{n\leq T\leq
n+1}\left|\frac{G_T}{T^{\delta}}\right|^q\right]\\&\leq&
\frac{1}{\varepsilon^q n^{q\delta}}E\left[\sup_{n\leq T\leq
n+1}\left|G_T\right|^q\right].
\end{eqnarray*}
Further, applying Garsia-Rodemich-Rumsey Lemma (see \cite[Lemma
A.3.1]{nualart-book}) for $\psi(x)=x^{q},\ p(x)=x^{\frac{m+2}{q}}$,
with $0<m<q\gamma-1$, we get for every $n\leq s, t\leq n+1$,
\begin{eqnarray*}\left|G_t-G_s\right|^q\leq C_{q,m}\left|t-s\right|^m
\int_n^{n+1}\int_n^{n+1}\frac{\left|G_u-G_v\right|^q}{\left|u-v\right|^{m+2}}dudv.\end{eqnarray*}
This together with $(\mathcal{A}_1)$ implies
\begin{eqnarray*}E\left[\sup_{n\leq s,t\leq
n+1}\left|G_t-G_s\right|^q\right]&\leq& C_{q,m}
\int_n^{n+1}\int_n^{n+1}\frac{\left|u-v\right|^{q\gamma}}{\left|u-v\right|^{m+2}}dudv\\
&=& C_{q,m}
\int_0^1\int_0^1\left|x-y\right|^{q\gamma-m-2}dxdy\\&\leq&
C_{\gamma,q,m}.\end{eqnarray*} Hence,
\begin{eqnarray*}E\left[\sup_{n\leq T\leq
n+1}\left|G_T\right|^q\right]&\leq&C_q\left( E\left[\sup_{n\leq
T\leq n+1}\left|G_T-G_n\right|^q\right] +E\left|G_n\right|^q\right)\\
&\leq&C_q\left( E\left[\sup_{n\leq s,t\leq
n+1}\left|G_t-G_s\right|^q\right] +n^{q\gamma}\right)
\\&\leq&
C_{\gamma,q,m}\left(1+n^{q\gamma}\right)
\\&\leq&
C_{\gamma,q,m} n^{q\gamma}.\end{eqnarray*} Then,
\begin{eqnarray*}
 P\left(\sup_{n\leq T \leq
n+1}\left|\frac{G_T}{T^{\delta}}\right|>\varepsilon\right) &\leq&
\frac{C_{\gamma,q,m}}{\varepsilon^q n^{q(\delta-\gamma)}}.
\end{eqnarray*}
 As a consequence,
since $p(\delta-\gamma)>1$, the above series converges, which proves
(\ref{cv-G}).\\
Now we prove (\ref{cv-quad-G}). Let $\gamma<\delta<1$,
\begin{eqnarray*}\frac{e^{-\theta T}}{T}\int_0^T|G_tX_t|dt
&\leq& \sup_{t\geq0}\left|\frac{G_tX_t}{t^{\delta}e^{\theta
t}}\right|\frac{e^{-\theta T}}{T}\int_0^Tt^{\delta}e^{\theta
t}dt\\&\leq& \sup_{t\geq0}\left|\frac{G_tX_t}{t^{\delta}e^{\theta
t}}\right|\frac{e^{-\theta T}}{T^{1-\delta}}\int_0^T e^{\theta
t}dt\\&\leq& \sup_{t\geq0}\left|\frac{G_tX_t}{t^{\delta}e^{\theta
t}}\right|\frac{e^{-\theta T}}{\theta T^{1-\delta}}
\\&&\longrightarrow0\end{eqnarray*}
almost surely as $T\rightarrow\infty$, where we used that
$\sup_{t\geq0}\left|\frac{G_tX_t}{t^{\delta}e^{\theta
t}}\right|<\infty$ almost surely, thanks to (\ref{X-CV}) and
(\ref{cv-G}). Thus the proof of (\ref{cv-quad-G}) is done.
\end{proof}

\begin{lemma}\label{a.s.CVs2}
Assume that $(\mathcal{A}_1)$ holds with   $\gamma\in(0,1)$. Then,
almost surely, as $T\longrightarrow \infty$,
\begin{eqnarray}
&&e^{-\theta T} X_T  \longrightarrow
  \mu +\zeta_{\infty},\label{X-CV}\\
 &&e^{-\theta T}\int_0^TX_sds  \longrightarrow
\frac{1}{\theta}\left(\mu+\zeta_{\infty}\right),\label{mean-X-CV}\\
 &&\frac{e^{-\theta T}}{T}\int_0^TsX_sds  \longrightarrow
\frac{1}{\theta}\left(\mu+\zeta_{\infty}\right),\label{mean-sX-CV}\\
 &&\frac{e^{-\theta T}}{T^{\delta}}\int_0^T|X_s|ds  \longrightarrow  0 \quad \mbox{for any
}\delta>0,\label{mean-CV}\\
 &&e^{-2\theta T}\int_0^TX_s^2ds  \longrightarrow
\frac{1}{2\theta}\left(\mu+\zeta_{\infty}\right)^2,\label{quad-CV}
\end{eqnarray}
where $\zeta _{\infty}$ is defined in Lemma \ref{lemma-zeta}.
\end{lemma}
\begin{proof} Notice first that the convergence (\ref{X-CV}) is a
direct consequence  of (\ref{GV-zeta}) and  (\ref{zeta-cv}).
\\
On the other hand, using (\ref{IBP}), (\ref{zeta-cv}) and
(\ref{cv-G}) we have almost surely as $T\rightarrow\infty$,
\begin{eqnarray}e^{-\theta T}\int_0^Te^{\theta s}Z_sds=\frac{Z_T}{\theta}-\frac{e^{-\theta T}}{\theta}\int_0^TG_sds\longrightarrow
\frac{\zeta_{\infty}}{\theta^2}.\label{mean-Z}
\end{eqnarray}
Combining (\ref{mean-Z}), (\ref{GV-zeta}) and
(\ref{zeta-decomposition}) we get almost surely as
$T\rightarrow\infty$,
\[e^{-\theta T}\int_0^TX_sds= e^{-\theta T}\int_0^T\left(\mu(e^{\theta s}-1)+G_s+\theta e^{\theta s}Z_s \right)ds \longrightarrow
\frac{1}{\theta}\left(\mu+\zeta_{\infty}\right),\] which proves
(\ref{mean-X-CV}). Similarly,   (\ref{IBP}), (\ref{zeta-cv}),
(\ref{cv-G}) and (\ref{mean-Z}) imply that  almost surely as
$T\rightarrow\infty$,
\begin{eqnarray*}e^{-\theta T}\int_0^Tse^{\theta s}Z_sds=\frac{Z_T}{\theta}-\frac{e^{-\theta T}}{\theta T}\int_0^TsG_sds
-\frac{e^{-\theta T}}{\theta T}\int_0^Te^{\theta
s}Z_sds\longrightarrow \frac{\zeta_{\infty}}{\theta^2}.
\end{eqnarray*}
Combined with (\ref{GV-zeta}) and (\ref{zeta-decomposition}),
straightforward calculation as above, leads to  (\ref{mean-sX-CV}).
\\ Using (\ref{X-CV}) we have almost surely as
$T\rightarrow\infty$,
\[\frac{e^{-\theta T}}{T^{\delta}}\int_0^T|X_s|ds\leq \sup_{t\geq0}\left|\frac{X_t}{e^{\theta
t}}\right|\frac{e^{-\theta T}}{T^{\delta}}\int_0^Te^{\theta s}ds
\longrightarrow 0,\] which implies (\ref{mean-CV}).\\
 From \cite[Lemma 2.2]{EEO} we have
\begin{equation*}\lim_{T\rightarrow\infty}e^{-2\theta
T}\int_0^T e^{2\theta s}\zeta_s^2
ds=\frac{1}{2\theta}\zeta_{\infty}^2\ \mbox{ almost surely.}
\end{equation*}
Moreover, similar arguments as in (\ref{mean-Z}) we deduce that
almost surely as $T\rightarrow\infty$,
\begin{eqnarray*}e^{-2\theta T}\int_0^Te^{2\theta s}Z_sds\longrightarrow
\frac{\zeta_{\infty}}{2\theta^2}.
\end{eqnarray*}
Combining these  these two latter convergences   with
(\ref{GV-zeta}), (\ref{zeta-cv}) and (\ref{cv-G}), we can deduce
(\ref{quad-CV}).
\end{proof}

\section{Asymptotic behavior of the LSEs}\label{sect-asymp}
We will now analyze the strong consistency and the joint asymptotic
distribution of the LSEs $(\widetilde{\theta}_T,\widetilde{\mu}_T)$
and $(\widetilde{\theta}_T,\widetilde{\alpha}_T)$.
\subsection{Strong consistency}
In this section we will prove the strong consistency of the
estimators $\widetilde{\theta}_T$,  $\widetilde{\mu}_T$ and
$\widetilde{\alpha}_T$.
\begin{theorem}
Assume that $(\mathcal{A}_1)$ holds  and let
$\widetilde{\theta}_{T}$ and $\widetilde{\mu}_{T}$  be given by
(\ref{est3-theta}) and (\ref{est3-mu}) for every $T\geq 0$. Then
\begin{eqnarray} \widetilde{\theta}_{T}\longrightarrow \theta,\label{consistency-theta}
 \end{eqnarray}
 and
 \begin{eqnarray}
 \widetilde{\mu}_{T}\longrightarrow \mu  \label{consistency-mu}
 \end{eqnarray}
almost  surely, as $T \rightarrow \infty$. As a consequence, we
deduce that the estimator $\widetilde{\alpha}_{T}$   given in
(\ref{est3-alpha}) is also strongly consistent, that is
 \begin{eqnarray*}
 \widetilde{\alpha}_{T}=\widetilde{\mu}_{T}\widetilde{\theta}_{T}\longrightarrow \alpha=\mu\theta  \label{consistency-alpha}
 \end{eqnarray*}
almost  surely, as $T \rightarrow \infty$.
\end{theorem}
\begin{proof}
Using   (\ref{est3-theta}) we get
\begin{eqnarray*}
\widetilde{\theta}_T &=&\frac{\frac12 \left(e^{-\theta
T}X_T\right)^2 -e^{-\theta T}X_T\frac{e^{-\theta T}}{T}\int_0^T X_s
d s}{e^{-2\theta T}\int_0^T X_s^2ds-\left(\frac{e^{-\theta
T}}{\sqrt{T}}\int_0^T X_s d s\right)^2}\\
& &\longrightarrow\theta, \quad   \mbox{almost  surely, as } T
\rightarrow \infty,
\end{eqnarray*}
where  the last convergence comes from the convergences
(\ref{X-CV}), (\ref{quad-CV}) and (\ref{mean-CV}). Thus the
convergence
(\ref{consistency-theta}) is obtained.\\
Let us now prove (\ref{consistency-mu}). It follows from
(\ref{est3-mu}) that $\widetilde{\mu}_T$   can be written as follows
\begin{equation*}
\widetilde{\mu}_T =\frac{e^{-\theta T}}{T}\left[\int_0^T
X_s^2ds-\frac{X_T}{2}\int_0^T
X_sds\right]\times\frac{1}{\frac12e^{-\theta T}X_T- \frac{e^{-\theta
T}}{T}\int_0^T X_s d s}.
\end{equation*}
According to the convergences (\ref{X-CV}) and (\ref{mean-CV}) we
have, almost  surely, as $T \rightarrow \infty$,
\[\frac{1}{\frac12e^{-\theta T}X_T- \frac{e^{-\theta
T}}{T}\int_0^T X_s d s}\longrightarrow
\frac{2}{\mu+\zeta_{\infty}}.\]
 Therefore, it remains to prove
\begin{equation}\frac{e^{-\theta T}}{T}\left[\int_0^T
X_s^2ds-\frac{X_T}{2}\int_0^T X_sds\right]\longrightarrow
\frac{\mu}{2}\left(\mu+\zeta_{\infty}\right)\label{numerator-mu-cv}\end{equation}
almost surely, as $T \rightarrow \infty$.\\
 Using the
formula (\ref{IBP}) and the equation (\ref{GV}), we have
\begin{eqnarray*}&&\int_0^T
X_s^2ds-\frac{X_T}{2}\int_0^T X_sds\\&=&\int_0^T X_s
d\Sigma_s-\frac{1}{2}\left(\mu\theta T+\theta
\Sigma_T+G_T\right)\Sigma_T\\
&=&\int_0^T \left(\mu\theta s+\theta \Sigma_s+G_s\right)
d\Sigma_s-\frac{\mu\theta}{2}  T\Sigma_T-\frac{\theta}{2}
\Sigma_T^2-\frac12G_T \Sigma_T\\
&=&  \mu\theta \int_0^TsX_sds+\frac{\theta}{2}
\Sigma_T^2+\int_0^TG_s d\Sigma_s-\frac{\mu\theta}{2}
T\Sigma_T-\frac{\theta}{2} \Sigma_T^2-\frac12G_T \Sigma_T
\\
&=&  \left(\mu\theta  \int_0^TsX_sds-\frac{\mu\theta }{2}
T\Sigma_T\right)+\left( \int_0^TG_s d\Sigma_s-\frac12G_T
\Sigma_T\right)
\\
&:=&I_T+J_T.
\end{eqnarray*}
Moreover, by L'H\^{o}pital's rule and (\ref{X-CV}) we have
\begin{eqnarray*}  \frac{e^{-\theta T}}{T}I_T&=&\frac{e^{-\theta T}}{T}\left(\mu\theta  \int_0^TsX_sds-\frac{\mu\theta }{2} T\Sigma_T\right)\\
&&\longrightarrow \frac{\mu }{2}\left(\mu+\zeta_{\infty}\right)
\end{eqnarray*}almost surely, as $T \rightarrow \infty$.\\
On the other hand, taking $\gamma<\delta<1$,
\begin{eqnarray*}
\frac{e^{-\theta T}}{T}|J_T|&=&\frac{e^{-\theta
T}}{T}\left|\int_0^TG_s d\Sigma_s-\frac12G_T
\Sigma_T\right|\\&=&\frac{e^{-\theta T}}{T}\left|\int_0^TG_sX_s d
s-\frac12G_T
\Sigma_T\right|\\&\leq&\frac32\left(\sup_{s\geq0}\left|\frac{G_s}{s^{\delta}}\right|\right)\frac{e^{-\theta
T}}{T^{1-\delta}}\int_0^T|X_s| ds\\
&\longrightarrow &0
\end{eqnarray*}almost surely, as $T \rightarrow \infty$, where we
used (\ref{cv-G}) and (\ref{mean-CV}).\\
 Consequently, the
convergence (\ref{numerator-mu-cv}) is proved.
  Thus the proof of the theorem is done.
 \end{proof}

\subsection{Asymptotic distribution }
In this section  the following assumptions are required:
\begin{itemize}
\item[$(\mathcal{A}_2)$] There exist      $\lambda_G>0$
and $ \eta\in(0,1)$ such that, as $T\rightarrow\infty$
\[\frac{E\left(G_T^2\right)}{T^{2\eta}} \longrightarrow
\lambda_G^2.\]
\item[$(\mathcal{A}_3)$] The limiting variance of $e^{-\theta T}\int_0^Te^{\theta s}dG_s$  exists as
$T\rightarrow\infty$ i.e., there exists a constant $\sigma_G>0$ such
that
\[\lim_{T\rightarrow\infty}E\left[\left(e^{-\theta T}\int_0^Te^{\theta s}dG_s\right)^2\right]\longrightarrow\sigma_G^2.\]
\item[$(\mathcal{A}_4)$] For all fixed $s\geq0$,
\begin{eqnarray*} \lim_{T\rightarrow\infty}E\left(G_se^{-\theta T}\int_0^Te^{\theta
r}dG_r\right)=0.
\end{eqnarray*}
\item[$(\mathcal{A}_5)$] For all fixed $s\geq0$,
\begin{eqnarray*}\lim_{T\rightarrow\infty}\frac{E\left(G_sG_T\right)}{T^{\eta}} =0,\quad \lim_{T\rightarrow\infty}E\left(\frac{G_T}{T^{\eta}}
e^{-\theta T}\int_0^Te^{\theta r}dG_r\right)=0.
\end{eqnarray*}
\end{itemize}

In order to   investigate  the asymptotic behavior in distribution
of the estimators $(\widetilde{\theta}_T,\widetilde{\mu}_T)$ and
$(\widetilde{\theta}_T,\widetilde{\alpha}_T)$, as
$T\rightarrow\infty$, we will need the following lemmas.
\begin{lemma}Assume that $(\mathcal{A}_1)$ holds and let $X$ be the process
given by (\ref{GV}). Then we have for every $T>0$,
\begin{eqnarray}
\frac12  X_T^2 -\frac{X_T}{T}\int_0^T X_t d t&=&\theta\left(
\int_0^T X_t^2dt-\frac{1}{T}\left(\int_0^T X_t d
t\right)^2\right)\nonumber\\
&&+\left(\mu +\theta   Z_T \right) \int_0^T e^{\theta t}
dG_t+R_T,\label{numerator-theta}
\end{eqnarray} where $Z_T$ is given in (\ref{zeta}), and the process $R_T$ is defined by
\begin{eqnarray*}R_T&:=&\frac12\left(\mu\theta T\right)^2+\frac12  G_T^2- \mu G_T
-\frac{(\mu\theta)^2T^2}{2}
-\frac{G_T}{T}\int_0^TX_tdt-\theta \int_0^T G_t^2 dt\\
&&+\theta^2 \int_0^T e^{-\theta t} G_t\int_0^t e^{\theta s} G_sdsdt.
\end{eqnarray*}
Moreover, as $T\longrightarrow\infty$,
\begin{eqnarray}e^{-\theta
T}R_T\longrightarrow0\label{cv-R}
\end{eqnarray}
almost surely.
\end{lemma}
\begin{proof}In order to prove (\ref{numerator-theta}) we first need  to introduce the following
processes, for every $t\geq0$,
\begin{eqnarray*}
 A_t:= \mu \left(e^{\theta
t}-1\right);\quad M_t:=\int_0^te^{\theta s}G_s ds.
\end{eqnarray*}
Thus,  by  (\ref{explicit-GV}) and (\ref{IBP}) we have
\begin{eqnarray}
X_t&=& \mu \left(e^{\theta t}-1\right)+e^{\theta
t}\int_0^te^{-\theta s}dG_s\nonumber\\
&=&A_t+G_t+\theta e^{\theta t}Z_t.\label{GV-Z}
\end{eqnarray}
On the other hand, using (\ref{GV-Sigma}), we obtain
\begin{eqnarray}\frac12  X_T^2 &=&\frac12\left(\mu\theta T\right)^2+\frac12 G_T^2
+\mu\theta T G_T+\mu\theta^2 T
\Sigma_T+\frac{\theta^2}{2}\Sigma_T^2+\theta \Sigma_T
G_T,\label{equaX1}
\end{eqnarray}
where, according to (\ref{IBP}) and (\ref{GV-Sigma}),
\begin{eqnarray}\frac{\theta^2}{2}\Sigma_T^2
&=&\theta^2\int_0^T\Sigma_td\Sigma_t\nonumber\\&=&\theta^2\int_0^T\Sigma_tX_tdt
\nonumber\\&=&\theta\int_0^T X_t^2dt-\mu\theta^2\int_0^T
td\Sigma_t-\theta\int_0^TG_tX_tdt\nonumber
\\&=&\theta\int_0^T X_t^2dt-\mu\theta^2 T\Sigma_T+ \mu\theta^2\int_0^T
\Sigma_tdt-\theta\int_0^TG_tX_tdt.\label{equaX2}
\end{eqnarray}
Further, by (\ref{GV-Z}) and (\ref{IBP}),
\begin{eqnarray}-\theta\int_0^TG_tX_tdt
&=&-\theta\int_0^TG_tA_tdt-\theta\int_0^TG_t^2
dt-\theta^2\int_0^TG_te^{\theta t}Z_tdt
\nonumber\\
&=&-\theta\int_0^TG_tA_tdt-\theta\int_0^TG_t^2
dt-\theta^2\left(Z_TM_T-\int_0^TM_tdZ_t\right)
\nonumber\\
&=&-\theta\int_0^TG_tA_tdt-\theta\int_0^TG_t^2 dt-\theta^2
Z_TM_T\nonumber\\&&+\theta^2\int_0^T e^{-\theta t} G_t\int_0^t
e^{\theta s} G_sdsdt.\label{equaX3}
\end{eqnarray}
Also, by (\ref{GV-Sigma}) and (\ref{GV-Z}),
\begin{eqnarray}
\theta \Sigma_T G_T&=&G_T\left(X_T-\mu\theta
T-G_T\right)\nonumber\\
&=&G_T\left( -\mu\theta T + A_T+\theta e^{\theta
T}Z_T\right).\label{equaX4}
\end{eqnarray}
Combining (\ref{equaX1}),  (\ref{equaX2}),  (\ref{equaX3}) and
(\ref{equaX4}), we can conclude that
\begin{eqnarray}\frac12  X_T^2 &=&\frac12\left(\mu\theta T\right)^2+\frac12 G_T^2
+\mu\theta^2 T \Sigma_T+\theta\int_0^T X_t^2dt-\mu\theta^2
T\Sigma_T+ \mu\theta^2\int_0^T
\Sigma_tdt\nonumber\\&&-\theta\int_0^TG_tA_tdt-\theta\int_0^T G_t^2
dt-\theta^2 Z_TM_T+\theta^2\int_0^T e^{-\theta t} G_t\int_0^t
e^{\theta s} G_sdsdt \nonumber\\&&+G_T\left( A_T+\theta e^{\theta
T}Z_T\right).\label{equaX5}
\end{eqnarray}
 Using (\ref{IBP}), we have
\begin{eqnarray}-\theta^2 Z_TM_T+\theta e^{\theta
T}Z_TG_T&=&-\theta Z_T\left(\theta M_T-e^{\theta
T}G_T\right) \nonumber\\
&=&-\theta Z_T\int_0^Te^{\theta t}dG_t,\label{equaX6}
\end{eqnarray}
and
\begin{eqnarray}G_TA_T-\theta\int_0^TG_tA_tdt&=&
- \mu G_T+\mu\theta\int_0^TG_tdt+ \mu e^{\theta
T}G_T-\mu\theta\int_0^TG_te^{\theta t}dt\nonumber\\&=& - \mu
G_T+\mu\theta\int_0^TG_tdt+ \mu \int_0^Te^{\theta
t}dG_t.\label{equaX7}
\end{eqnarray}
Now, combining (\ref{equaX5}),   (\ref{equaX6}) and  (\ref{equaX7}),
we obtain
\begin{eqnarray}\frac12  X_T^2 &=&\frac12\left(\mu\theta T\right)^2+\frac12 G_T^2
+\theta\int_0^T X_t^2dt+ \mu\theta^2\int_0^T
\Sigma_tdt\nonumber\\&&-\theta\int_0^TG_t^2 dt+\theta^2\int_0^T
e^{-\theta t} G_t\int_0^t e^{\theta s} G_sdsdt \nonumber\\&&-\theta
Z_T\int_0^Te^{\theta t}dG_t- \mu G_T+\mu\theta\int_0^TG_tdt+ \mu
\int_0^Te^{\theta t}dG_t.\label{equaX8}
\end{eqnarray}
On the other hand, using (\ref{GV-Sigma}),
\begin{eqnarray}
 -\frac{X_T}{T}\int_0^T X_t d t=-\mu\theta\Sigma_T-\frac{\theta}{T}\Sigma_T^2-\frac{G_T}{T}\Sigma_T.\label{equaX9}
\end{eqnarray}
  Combining  (\ref{equaX8}),  (\ref{equaX9}) and the fact that
  \[\mu\theta\int_0^T
\Sigma_tdt-\mu\theta  \Sigma_T+\mu\theta \int_0^T
G_tdt=-\frac{(\mu\theta T)^2}{2},
\]we get therefore
  (\ref{numerator-theta}).\\
  Finally, the convergence (\ref{cv-R}) is a direct consequence of (\ref{cv-G}) and
  (\ref{mean-X-CV}).
\end{proof}

\begin{lemma}\label{lemma-cv-couple-law}
   Let $F$ be any
$\sigma\{G_t,t\geq0\}-mesurable$ random variable such that
$P(F<\infty)=1$. If   $(\mathcal{A}_1)-$$(\mathcal{A}_4)$ hold, then
as $T\rightarrow\infty$
\begin{eqnarray}\left(F,e^{-\theta T}\int_0^Te^{\theta
t}dG_t\right)\overset{Law}{\longrightarrow}\left(F,\sigma_G
N_2\right).\label{cv-law-couple-F}\end{eqnarray} Moreover, if
$(\mathcal{A}_5)$ holds, then as $T\rightarrow\infty$
\begin{eqnarray}\left(\frac{G_T}{T^{\eta}},F,e^{-\theta T}\int_0^Te^{\theta
t}dG_t\right)\overset{Law}{\longrightarrow}\left(\lambda_G
N_1,F,\sigma_G N_2\right),\label{cv-law-couple-F-G}\end{eqnarray}
where $N_1\sim\mathcal{N}(0,1)$, $N_2\sim\mathcal{N}(0,1)$ and $G$
are independent.
\end{lemma}
\begin{proof} The convergence (\ref{cv-law-couple-F}) is proved in
\cite[Lemma 2.4]{EEO}.  Now we prove  (\ref{cv-law-couple-F-G}).
Using similar arguments as in the proof of   \cite[Lemma 7]{EN} it
suffices to prove that for every positive integer $d$, and positive
constants $s_1,\ldots,s_d$,
\begin{eqnarray*}\left(\frac{G_T}{T^{\eta}},G_{s_1},\ldots,G_{s_d},e^{-\theta T}\int_0^Te^{\theta
t}dG_t\right)\overset{Law}{\longrightarrow}\left(\lambda_G
N_1,G_{s_1},\ldots,G_{s_d},\sigma_G N_2\right)\end{eqnarray*} as
 $T\rightarrow\infty$. Moreover, since the left-hand side in this latter convergence is a
Gaussian vector,   it is sufficient to establish the convergence of
its covariance matrix. Therefore, using
 $(\mathcal{A}_2)-$$(\mathcal{A}_5)$, the desired result is
 obtained.

\end{proof}

Recall that if $X\sim \mathcal{N}(m_1,\sigma_1)$ and $Y\sim
\mathcal{N}(m_2,\sigma_2)$ are two independent random variables,
then $X/Y$ follows a   Cauchy-type distribution. For a motivation
and further references, we refer the reader to \cite{PTM}, as well
as \cite{marsaglia}. Notice also that if $N\sim\mathcal{N}(0,1)$ is
independent of $G$, then $N$ is independent of $\zeta_{\infty}$,
since $\zeta_\infty =\theta \int_0^{\infty}e^{-\theta s}G_sds$ is a
functional of $G$.

\begin{theorem}
Assume that $(\mathcal{A}_1)-$$(\mathcal{A}_4)$ hold. Suppose that
$N_1\sim\mathcal{N}(0,1)$, $N_2\sim\mathcal{N}(0,1)$ and $G$ are
independent. Then as
 $T\rightarrow\infty$,
\begin{eqnarray}e^{\theta T}(\widetilde{\theta}_T-\theta)\overset{Law}{\longrightarrow}\frac{2\theta\sigma_GN_2}{\mu
+\zeta_{\infty}},\label{cv-theta-law}
\end{eqnarray}
\begin{eqnarray}T^{1-\eta}\left(\widetilde{\mu}_T-\mu
\right)\overset{Law}{\longrightarrow}
\frac{\lambda_G}{\theta}N_1,\label{cv-mu-law}
\end{eqnarray}
\begin{eqnarray}T^{1-\eta}\left(\widetilde{\alpha}_T-\alpha
\right)\overset{Law}{\longrightarrow} \lambda_G
N_1.\label{cv-alpha-law}
\end{eqnarray}
 Moreover, if  $(\mathcal{A}_5)$ holds, then as
 $T\rightarrow\infty$,\begin{eqnarray}\left(e^{\theta T}(\widetilde{\theta}_T-\theta),T^{1-\eta}\left(\widetilde{\mu}_T-\mu
\right)\right)\overset{Law}{\longrightarrow}\left(\frac{2\theta\sigma_GN_2}{\mu
+\zeta_{\infty}},\frac{\lambda_G}{\theta}N_1\right),\label{cv-joint-theta-mu-law}
\end{eqnarray}
\begin{eqnarray}\left(e^{\theta
T}(\widetilde{\theta}_T-\theta),T^{1-\eta}\left(\widetilde{\alpha}_T-\alpha
\right)\right)\overset{Law}{\longrightarrow}\left(\frac{2\theta\sigma_GN_2}{\mu
+\zeta_{\infty}},\lambda_G
N_1\right).\label{cv-joint-theta-alpha-law}
\end{eqnarray}
\end{theorem}

\begin{proof}
First we prove (\ref{cv-theta-law}). From (\ref{est3-theta}) and
(\ref{numerator-theta}) we can write
\begin{eqnarray}
e^{\theta T}\left(\widetilde{\theta}_T-\theta\right)&=&\frac{\left(
\mu  +\theta   Z_T \right) e^{-\theta T}\int_0^T e^{\theta t} dG_t+
e^{-\theta T}R_T}{e^{-2\theta T}\left( \int_0^T
X_t^2dt-\frac{1}{T}\left(\int_0^T X_t d t\right)^2\right)}\nonumber\\
&=&\frac{ e^{-\theta T}\int_0^T e^{\theta t} dG_t}{\left( \mu
+\zeta_{\infty}\right)}\times \frac{\left( \mu
+\zeta_{\infty}\right)\left( \mu  +\theta Z_T \right)}{e^{-2\theta
T}\left( \int_0^T X_t^2dt-\frac{1}{T}\left(\int_0^T X_t d
t\right)^2\right)}\nonumber\\&&+ \frac{e^{-\theta T}R_T}{e^{-2\theta
T}\left( \int_0^T X_t^2dt-\frac{1}{T}\left(\int_0^T X_t d
t\right)^2\right)}\nonumber
\\
&=:&a_T\times b_T+c_T.\label{decomp-theta-estim}
\end{eqnarray}
Lemma \ref{lemma-cv-couple-law}  yields , as $T\rightarrow\infty$,
\[a_T\overset{ Law}{\longrightarrow}\frac{\sigma_GN_2}{ \mu  +\zeta_{\infty}},\]
 whereas (\ref{quad-CV}), (\ref{mean-CV}) and (\ref{zeta-cv}) imply that $b_T
\longrightarrow 2\theta$ almost surely as $T\rightarrow\infty$. On
the other hand, by (\ref{quad-CV}), (\ref{mean-CV}) and
(\ref{cv-R}), we obtain that $c_T \longrightarrow 0$ almost surely
as
$T\rightarrow\infty$.\\
Combining all these facts together with (\ref{cv-law-couple-F}), we
get (\ref{cv-theta-law}).
\\
Using (\ref{est3-theta}) and (\ref{est3-mu}),  a straightforward
calculation shows that $\widetilde{\theta}_T$ and
$\widetilde{\mu}_T$ verify
\begin{eqnarray}
 \widetilde{\theta}_T\widetilde{\mu}_T T&=& \widetilde{\theta}_T\frac{\widetilde{\mu}_T}{X_T} X_TT\nonumber\\
 &=& X_T-\widetilde{\theta}_T\int_0^T X_t dt.\label{decomp1}
\end{eqnarray}
Combining (\ref{decomp1}) with (\ref{GV}), we obtain
\begin{eqnarray}T^{1-\eta}\left(\widetilde{\mu}_T-\mu \right)
&=&\frac{1}{\widetilde{\theta}_T}\left[-e^{\theta
T}\left(\widetilde{\theta}_T-\theta\right)\frac{e^{-\theta
T}}{T^{\eta}}\int_0^T X_t dt-\mu\frac{T^{1-\eta}}{e^{\theta T}}
e^{\theta T}\left(\widetilde{\theta}_T-\theta\right)
+\frac{G_T}{T^{\eta}}\right]\nonumber\\
&=:&\frac{1}{\widetilde{\theta}_T}\left[D_T
+\frac{G_T}{T^{\eta}}\right].\label{decomp-mu-estim}\end{eqnarray}
 Using (\ref{mean-CV}),
(\ref{consistency-theta}),  (\ref{cv-theta-law}) and Slutsky's
theorem, we obtain $D_T\longrightarrow0$ in
probability as $T\rightarrow\infty$.  This together with $(\mathcal{A}_2)$  implies (\ref{cv-mu-law}).\\
 Further, according to (\ref{decomp1}) and (\ref{GV}) we can write
\begin{eqnarray*}T^{1-\eta}\left(\widetilde{\alpha}_T-\alpha \right)
&=&-e^{\theta
T}\left(\widetilde{\theta}_T-\theta\right)\frac{e^{-\theta
T}}{T^{\eta}}\int_0^T X_t dt +\frac{G_T}{T^{\eta}},\end{eqnarray*}
which proves (\ref{cv-alpha-law}), by using (\ref{mean-CV}),
  (\ref{cv-theta-law}), $(\mathcal{A}_2)$,
and Slutsky's theorem.
\\
Let us now do the proof of (\ref{cv-joint-theta-mu-law}). By
(\ref{decomp-theta-estim}) and (\ref{decomp-mu-estim}) we have
\begin{eqnarray*}\left(e^{\theta T}(\widetilde{\theta}_T-\theta),T^{1-\eta}\left(\widetilde{\mu}_T-\mu
\right)\right)&=&\left(a_T\times
b_T+c_T,\frac{1}{\widetilde{\theta}_T}\left[D_T
+\frac{G_T}{T^{\eta}}\right]\right)
\\&=&\left(a_T\times
b_T,\frac{1}{\widetilde{\theta}_T}\frac{G_T}{T^{\eta}}\right)+\left(c_T,\frac{1}{\widetilde{\theta}_T}D_T\right)
\\&=&\frac{2\theta^2}{\widetilde{\theta}_T}\left(a_T,\frac{G_T}{T^{\eta}}\right)+\frac{1}{\widetilde{\theta}_T}\left(a_T\times
(b_T\times\widetilde{\theta}_T-2\theta^2),0\right)
\\&&\quad+\left(c_T,\frac{1}{\widetilde{\theta}_T}D_T\right).
\end{eqnarray*}
By the above convergences and  Slutsky's theorem we  deduce
\[\frac{1}{\widetilde{\theta}_T}\left(a_T\times
(b_T\times\widetilde{\theta}_T-2\theta^2),0\right)\longrightarrow0,
\ \left(c_T,\frac{1}{\widetilde{\theta}_T}D_T\right)\longrightarrow0
\mbox{ in probability as }T\rightarrow\infty.\] Therefore, using
(\ref{cv-law-couple-F-G}), we
 obtain as $T\rightarrow\infty$,
 \begin{eqnarray*}\frac{1}{\widetilde{\theta}_T}\left(2\theta^2a_T,\frac{G_T}{T^{\eta}}\right)&=&
 \frac{1}{\widetilde{\theta}_T}\left(2\theta^2\frac{ e^{-\theta T}\int_0^T e^{\theta t} dG_t}{\left( \mu
+\zeta_{\infty}\right)},\frac{G_T}{T^{\eta}}\right)\\&& \overset{
Law}{\longrightarrow}\left(\frac{2\theta\sigma_G N_2}{ \mu
+\zeta_{\infty}},\frac{\lambda_G}{\theta} N_1\right),
 \end{eqnarray*}
 which completes the proof of (\ref{cv-joint-theta-mu-law}). Finally, following the same arguments as  in the proof of
(\ref{cv-joint-theta-mu-law}) we can deduce
(\ref{cv-joint-theta-alpha-law}).
\end{proof}
 \begin{remark}Following  \cite{EN,BEO}, the author of \cite{Yu} considered
 the above
problem  estimation for (\ref{GV}) in the case when $G$ is  a
self-similar Gaussian process with index $L\in(\frac12,1)$. The
author studied the asymptotic behavior for each component of
$(\widetilde{\theta}_T,\widetilde{\mu}_T)$ separately. In the
present work  we provide a general technique that   allows to extend
the results of \cite{Yu}   for all general  index $L\in(0,1)$.
Moreover, we study the joint asymptotic behavior of
$(\widetilde{\theta}_T,\widetilde{\mu}_T)$. \end{remark}

\section{Applications to Gaussian Vasicek
processes}\label{sect-appli}
This section is devoted to some examples of   Gaussian Vasicek
processes. We will discuss  the following three cases of the driving
Gaussian process $G$ of (\ref{GV}): fractional Brownian motion,
subfractional Brownian motion and bifractional Brownian motion. We
will need the following technical lemma.
\begin{lemma}\label{key-applications}For every $H\in(0,1)$, we have
\begin{eqnarray*}  k_H&:=&\int_0^{\infty}\int_0^{\infty} e^{-\theta t} e^{-\theta
s}t^{2H}dsdt=\frac{\Gamma(2H+1)}{\theta^{2H+2}},\label{equa-a}\\
 l_H&:=&\int_0^{\infty}\int_0^{\infty} e^{-\theta t} e^{-\theta
s}|t-s|^{2H}dsdt=\frac{\Gamma(2H+1)}{\theta^{2H+2}},\label{equa-b}\\
 m_H&:=&\int_0^{\infty}\int_0^{\infty} e^{-\theta t} e^{-\theta
s}|t+s|^{2H}dsdt=\frac{\Gamma(2H+2)}{\theta^{2H+2}}.\label{equa-d}
\end{eqnarray*}
\end{lemma}
\begin{proof}We have
\begin{eqnarray*}  k_H&=&\int_0^{\infty}\int_0^{\infty} e^{-\theta t} e^{-\theta
s}t^{2H}dsdt\\
&=&\frac{1}{\theta}\int_0^{\infty} e^{-\theta t}
t^{2H}dt\\&=&\frac{\Gamma(2H+1)}{\theta^{2H+2}}.
\end{eqnarray*}
Further, we have
\begin{eqnarray*}
 l_H&=&\int_0^{\infty}\int_0^{\infty} e^{-\theta t} e^{-\theta
s}|t-s|^{2H}dsdt\\&=& 2\int_0^{\infty}\int_0^{t} e^{-\theta t}
e^{-\theta s}(t-s)^{2H}dsdt\\&=& 2\int_0^{\infty}e^{-2\theta
t}\int_0^{t}
 e^{\theta u}u^{2H}dudt
 \\&=& 2\int_0^{\infty}e^{\theta u}u^{2H}\int_u^{\infty}e^{-2\theta t}
 dtdu
 \\&=& \frac{1}{\theta}\int_0^{\infty}e^{-\theta u}u^{2H}du\\
&=&\frac{\Gamma(2H+1)}{\theta^{2H+2}}.
\end{eqnarray*}
Finally,  we have
\begin{eqnarray*}
 m_H&=&\int_0^{\infty}\int_0^{\infty} e^{-\theta t} e^{-\theta
s}(t+s)^{2H}dsdt\\&=&\int_0^{\infty}\int_t^{\infty} e^{-\theta u}
u^{2H}dudt
\\&=&\int_0^{\infty} e^{-\theta u}
u^{2H+1}du\\
&=&\frac{\Gamma(2H+2)}{\theta^{2H+2}}.
\end{eqnarray*}
\end{proof}

\subsection{Fractional Vasicek process}\label{section-fBmV}
The  fractional Brownian motion (fBm) $B^H:= \left\{B_t^H, t\geq
0\right\}$ with Hurst parameter $H\in(0,1)$, is defined as a
centered Gaussian process starting from zero with covariance
\[E\left(B^H_tB^H_s\right)=\frac{1}{2}\left(t^{2H}+s^{2H}-|t-s|^{2H}\right).\]
Note that, when $H=\frac12$, $B^{\frac12}$ is a standard Brownian
motion. \\
We have
\begin{eqnarray}E\left[\left(B^H_t-B^H_s\right)^2\right]=|s-t|^{2H};\quad s,\
t\geq~0. \label{quasi-helix-fBm}\end{eqnarray} Let us first start
with the following simulated path of the fractional Vasicek process,
i.e., when $G=B^H$ in (\ref{GV}),
 \begin{eqnarray}X_0=0;\quad dX_t=\theta\left(\mu+X_t\right)dt+dB^H_t.
 \label{fBmV}\end{eqnarray}
\begin{itemize}
    \item First, we generate the fractional Brownian motion using the wavelet method (see \cite{AS}).
    \item After that we simulate the process (\ref{fBmV}) using the Euler-Maruyama method for
    different values of $H$, $\theta$ and $\mu$ (see Figure \ref{fvp}).
\end{itemize}
We simulate a sample path on the interval $[0,1]$ using a regular
partition of 10,000 intervals.

\begin{figure}[H]
    \caption{The sample path of Fractional Vasicek process. }
    \label{fvp}
    \begin{center}
        {\includegraphics[ width=3.5cm]{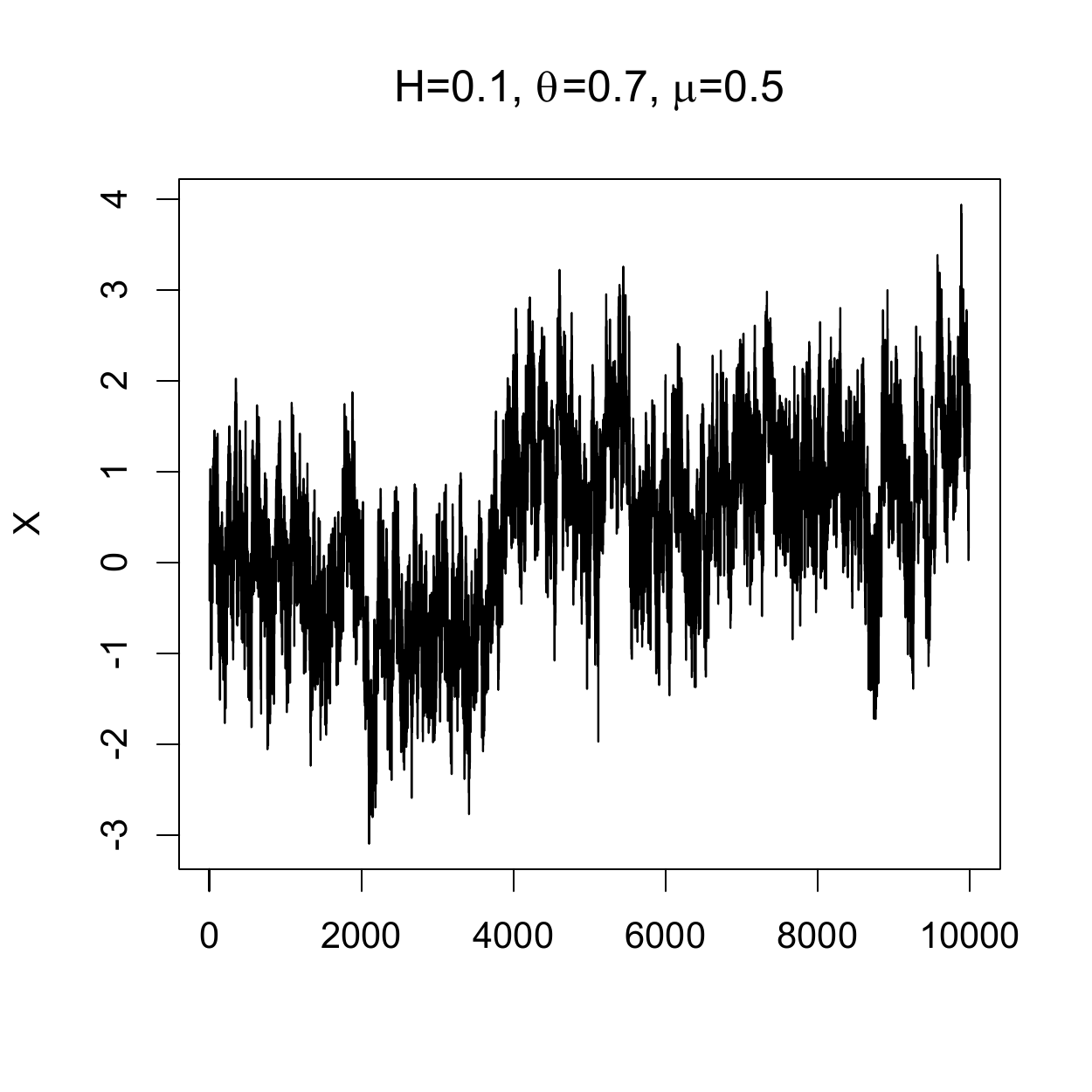}}
        {\includegraphics[ width=3.5cm]{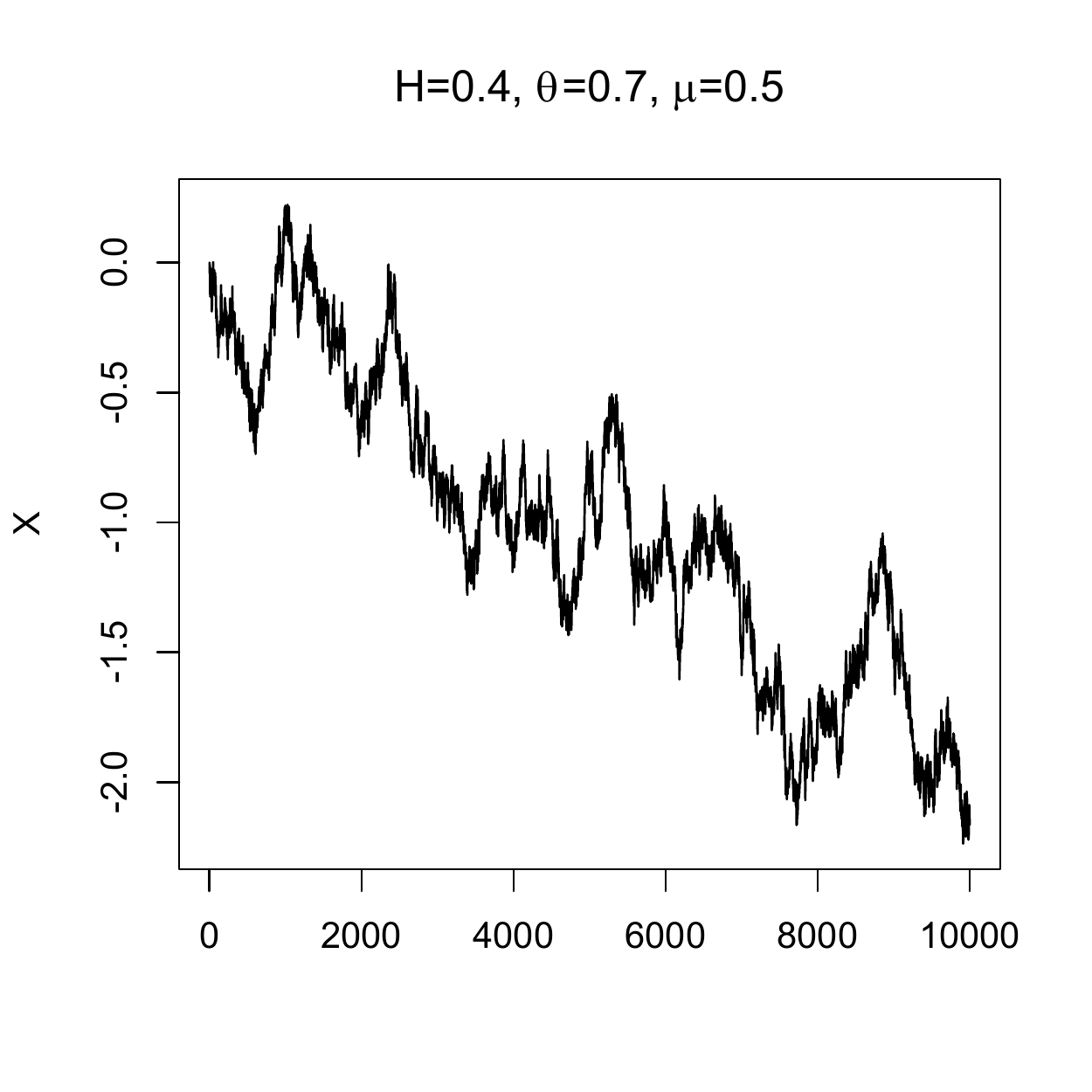}}
        {\includegraphics[ width=3.5cm]{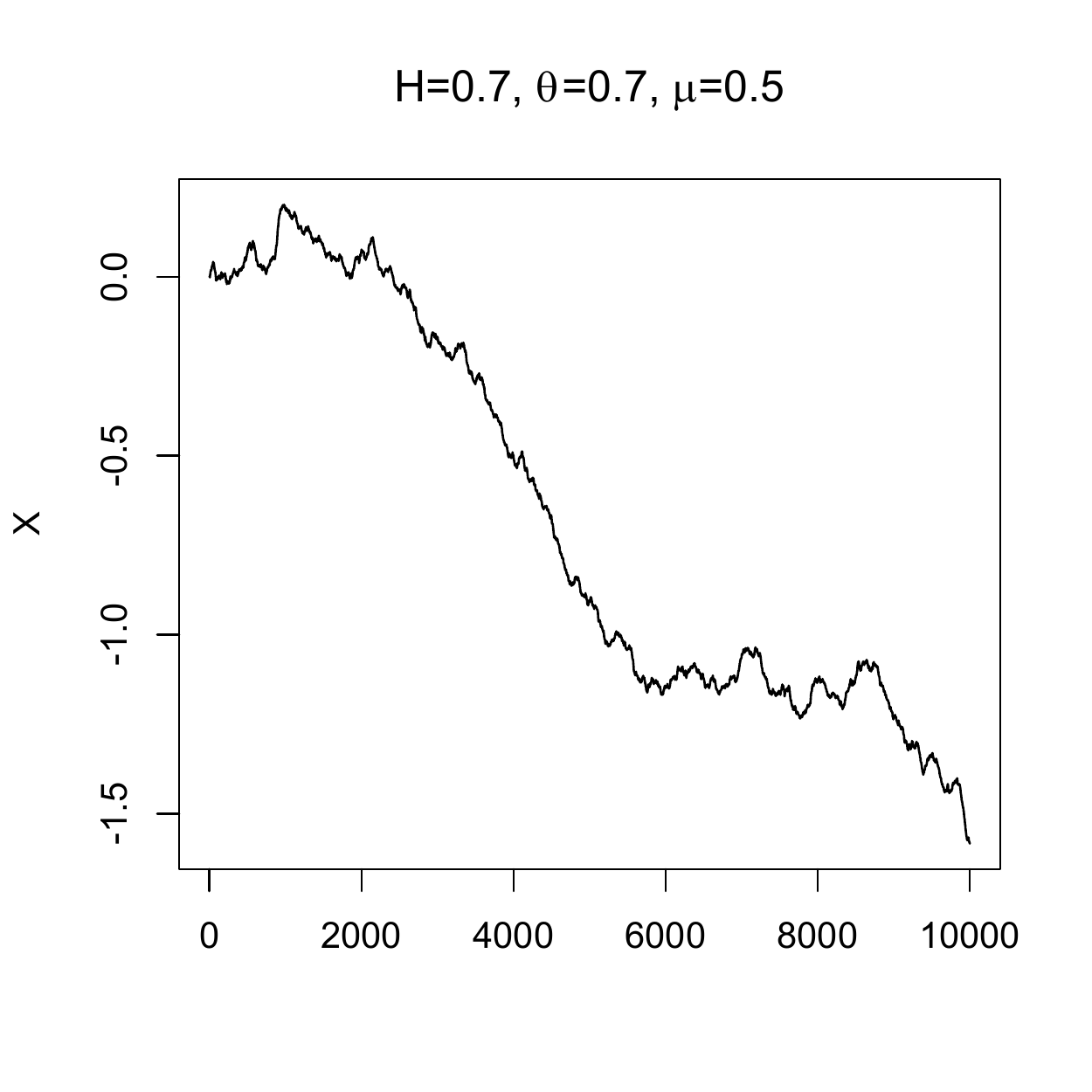}}
        {\includegraphics[ width=3.5cm]{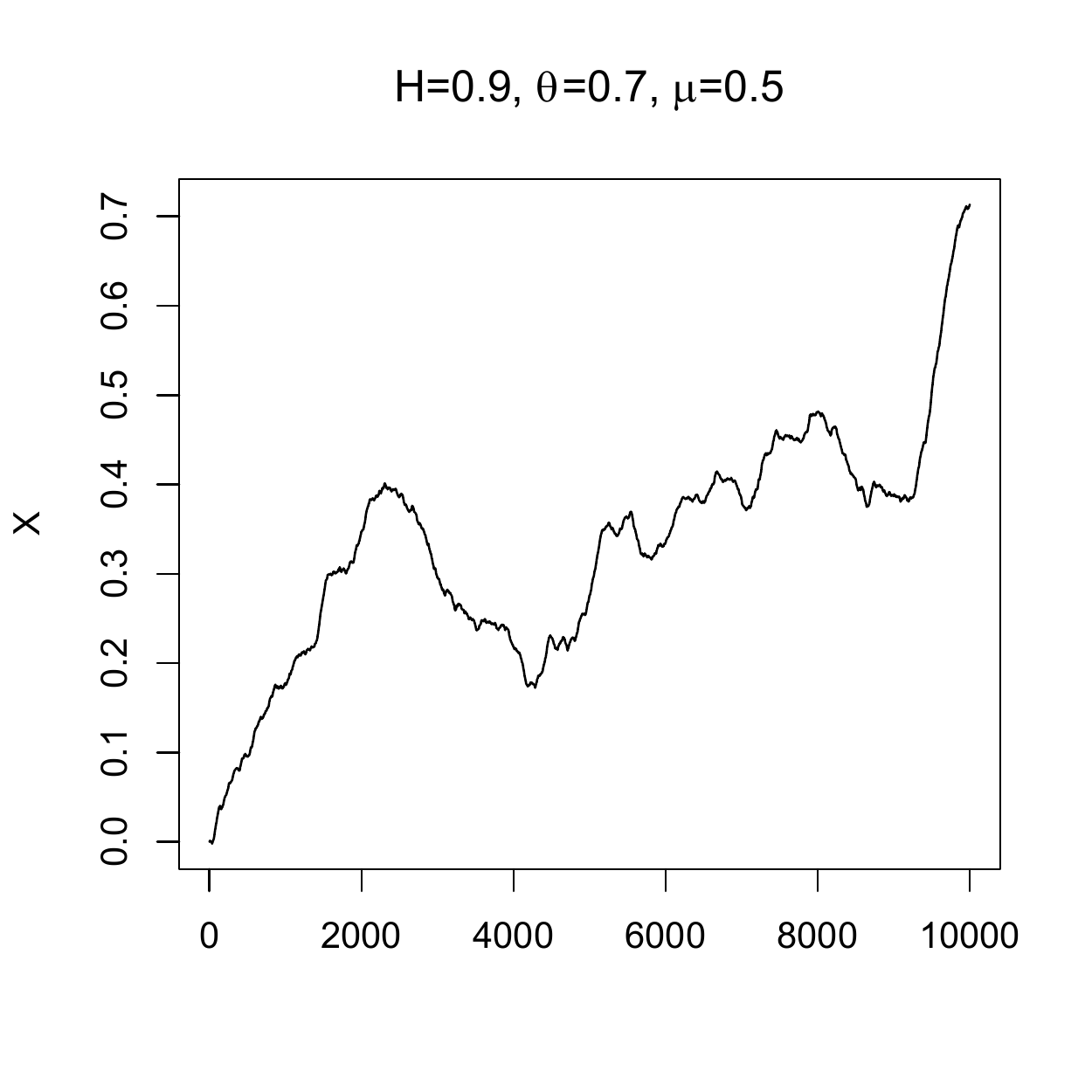}}
    \end{center}
\end{figure}

\noindent   Let us now discuss the asymptotic behavior of
$\widetilde{\theta}_{T}$  and $\widetilde{\mu}_{T}$.
 Thanks to (\ref{quasi-helix-fBm}), the process $B^H$ satisfies the assumptions $(\mathcal{A}_1)$ and $(\mathcal{A}_2)$ for $G=B^H$ and
 $\gamma=\eta=H$.\\
Moreover, by \cite[Proposition 3.1.]{EEO}, the assumptions
$(\mathcal{A}_3)$ and $(\mathcal{A}_4)$ hold with
\begin{eqnarray}\sigma_{B^H}^2=\frac{H\Gamma(2H)}{\theta^{2H}},\quad
 \lambda_{B^H}= 1.\label{sigma-fBm}\end{eqnarray}
 It remains to check the assumption $(\mathcal{A}_5)$ for $G=B^H$ and $\eta=H$. Fix
 $s\geq0$, using Taylor's expansion
 we have as $T\rightarrow\infty$,
\begin{eqnarray*}\frac{E\left(B^H_sB^H_T\right)}{T^{H}}
&=&\frac{1}{2T^{H}}\left(s^{2H}+T^{2H}-(T-s)^{2H}\right)\\
&=&\frac{s^{2H}}{2T^{H}}+\frac{T^{H}}{2}\left(1-(1-\frac{s}{T})^{2H}\right)
\\
&=&\frac{s^{2H}}{2T^{H}}-\frac{T^{H}}{2}\left(2H\frac{s}{T}+o(1/T)\right)
\\
&=&\frac{s^{2H}}{2T^{H}}-\frac{T^{H-1}}{2}\left(2Hs+o(1)\right)
\longrightarrow0.
\end{eqnarray*}
Furthermore, by   (\ref{IBP}) and the usual integration by parts
formula, we obtain
\begin{eqnarray*}E\left[\frac{B^H_T}{T^{H}}
e^{-\theta T}\int_0^Te^{\theta
r}dB^H_r\right]&=&E\left[\frac{B^H_T}{T^{H}} \left(B^H_T-\theta
e^{-\theta T} \int_0^Te^{\theta
t}B^H_tdt\right)\right]\\&=&\frac{1}{T^{H}}\left[
T^{2H}-\frac{\theta}{2} e^{-\theta T} \int_0^Te^{\theta
t}\left(T^{2H}+t^{2H}-(T-t)^{2H}\right)dt\right]
\\&=&\frac{1}{T^{H}}\left[
T^{2H}e^{-\theta T}+H e^{-\theta T} \int_0^Tt^{2H-1}e^{\theta t}
dt+\frac{\theta}{2} e^{-\theta T}\int_0^Te^{\theta
t}(T-t)^{2H}dt\right]\\&&\longrightarrow0 \mbox{ as
$T\rightarrow\infty$},
\end{eqnarray*} where we used that as
$T\rightarrow\infty$,
\[\frac{e^{-\theta T}}{T^{H}}\int_0^Te^{\theta
t}(T-t)^{2H}dt=\frac{1}{T^{H}}\int_0^Te^{-\theta
x}x^{2H}dx\longrightarrow0,\] and using  L'H\^{o}pital's rule,
\begin{eqnarray*} \lim_{T\rightarrow\infty}\frac{e^{-\theta T}}{T^{H}} \int_0^Tt^{2H-1}e^{\theta t}
dt=\lim_{T\rightarrow\infty}\frac{T^{2H-1}}{\theta T^{H}+HT^{H-1}}
=\lim_{T\rightarrow\infty}\frac{T^{H-1}}{\theta +\frac{H}{T}}=0.
\end{eqnarray*}
Thus the assumption $(\mathcal{A}_5)$ holds.\\
 On the
other hand, by Lemma \ref{key-applications}, we have for every
$H\in(0,1)$, the variance of $\zeta_{B^H,\infty}:=\theta
\int_0^{\infty}e^{-\theta s}B^H_sds$, given in (\ref{zeta-cv}) when
$G=B^H$, is equal to
\begin{eqnarray*}
E(\zeta_{B^H,\infty}^2)&=&\theta^2
\int_0^{\infty}\int_0^{\infty}e^{-\theta s}e^{-\theta
t}E(B^H_sB^H_t)dsdt\nonumber\\
&=&\theta^2 \left(k_H-\frac12 l_H\right)\nonumber\\
&=&\frac{H\Gamma(2H)}{\theta^{2H}},\label{zeta-fBm}
\end{eqnarray*}
which is equal in this case to $\sigma_{B^H}^2$, given in
(\ref{sigma-fBm}). \\Thus, we obtain the following result.
\begin{proposition}
Assume that $H\in(0,1)$ and the process $G$, given in (\ref{GV}), is
a fBm $B^H$. Then, almost surely, as $T\rightarrow\infty$,
\begin{eqnarray*} \left(\widetilde{\theta}_{T},\widetilde{\mu}_{T}\right)\longrightarrow \left(\theta,\mu\right),
\quad
\left(\widetilde{\theta}_{T},\widetilde{\alpha}_{T}\right)\longrightarrow
\left(\theta,\alpha\right).
 \end{eqnarray*}
 In addition, if
$N_1\sim\mathcal{N}(0,1)$, $N_2\sim\mathcal{N}(0,1)$ and $B^H$ are
independent, then as $T\rightarrow\infty$,
\begin{eqnarray*}\left(e^{\theta T}(\widetilde{\theta}_T-\theta),T^{1-H}\left(\widetilde{\mu}_T-\mu
\right)\right)\overset{Law}{\longrightarrow}\left(\frac{2\theta\sigma_{B^H}N_2}{\mu
+\zeta_{B^H,\infty}},\frac{1}{\theta}N_1\right),
\end{eqnarray*}
\begin{eqnarray*}\left(e^{\theta
T}(\widetilde{\theta}_T-\theta),T^{1-H}\left(\widetilde{\alpha}_T-\alpha
\right)\right)\overset{Law}{\longrightarrow}\left(\frac{2\theta\sigma_{B^H}N_2}{\mu
+\zeta_{B^H,\infty}}, N_1\right),
\end{eqnarray*}
where $\sigma_{B^H}$ is defined in (\ref{sigma-fBm}), and
$\zeta_{B^H,\infty}\sim \mathcal{N}(0,\sigma_{B^H}^2)$ is
independent of $N_1$ and $N_2$.
\end{proposition}

\subsection{Subfractional  Vasicek process}  \label{section-subfBmV}
The subfractional Brownian motion (subfBm)  $S^H:= \left\{S_t^H,
t\geq 0\right\}$ with parameter $H\in(0, 1)$   is a centered
Gaussian process with covariance function
\[E\left(S^H_tS^H_s\right)=t^{2H}+s^{2H}-\frac{1}{2}\left((t+s)^{2H}+|t-s|^{2H}\right);\quad s,\ t\geq0.\]
In particular, for every $T\geq0$,
\[\frac{E[(S^H_T)^2]}{T^{2H}}=2-2^{2H-1}.\]
 Note that, when $H=\frac12$, $S^{\frac12}$ is a
standard Brownian motion.  Moreover, it is known that
$$E\left[\left(S^H_t-S^H_s\right)^2\right]\leq
(2-2^{2H-1})|s-t|^{2H};\quad s,\ t\geq~0.$$ Let us first start with
the following simulated path of the subfractional Vasicek process,
i.e., when $G=S^H$ in (\ref{GV}),
\begin{eqnarray}X_0=0;\quad dX_t=\theta\left(\mu+X_t\right)dt+dS^H_t.
\label{sfBmV}\end{eqnarray}
\begin{itemize}
    \item First, we generate the subfractional Brownian motion using \cite{MF}.
    \item After that we simulate the process (\ref{sfBmV}) using the Euler-Maruyama method for
    different values of $H$, $\theta$ and $\mu$ (see Figure \ref{sfvp}).
\end{itemize}
We simulate a sample path on the interval $[0,1]$ using a regular
partition of 10,000 intervals.

\begin{figure}[H]
    \caption{The sample path of subfractional Vasicek process. }
    \label{sfvp}
    \begin{center}
        {\includegraphics[ width=3.5cm]{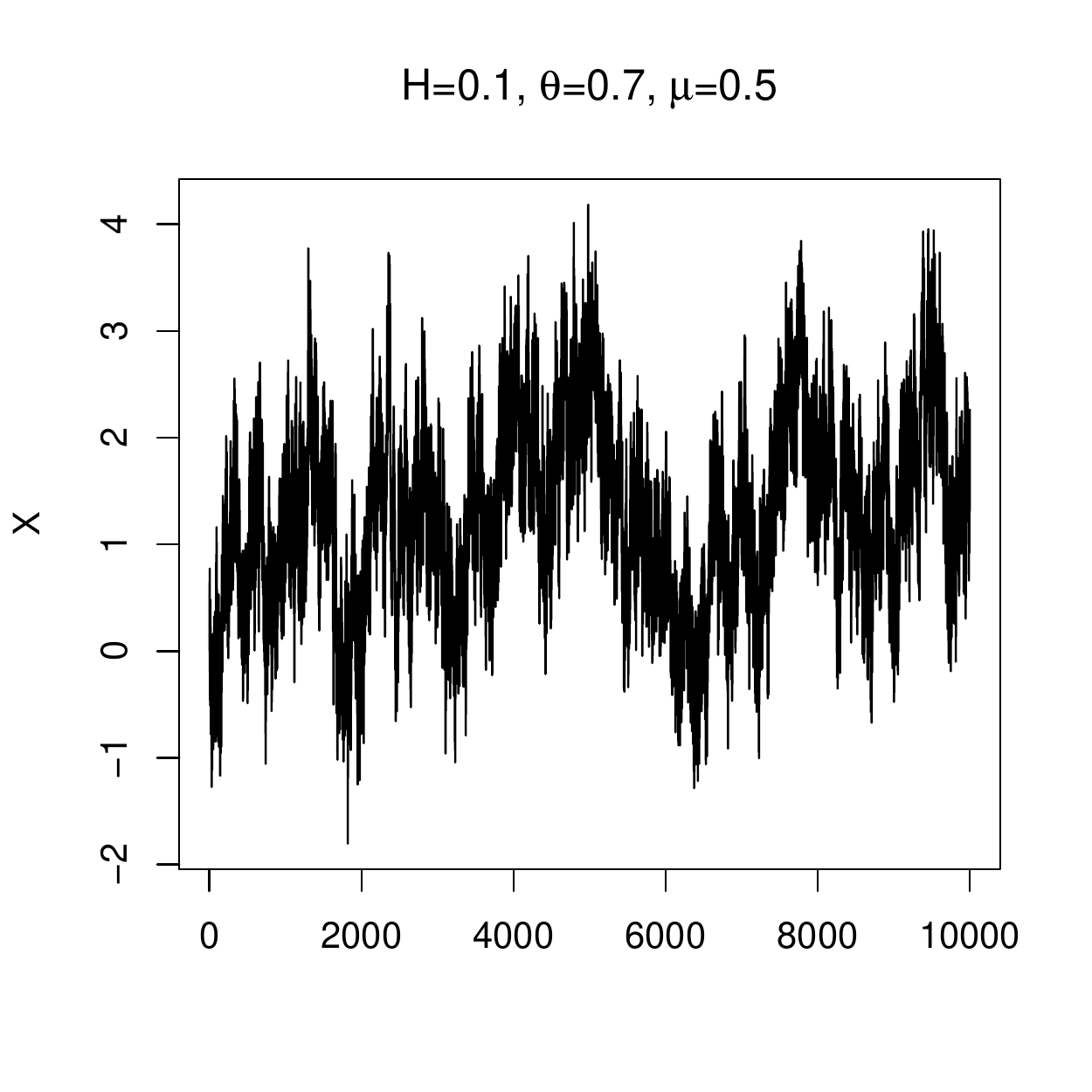}}
        {\includegraphics[ width=3.5cm]{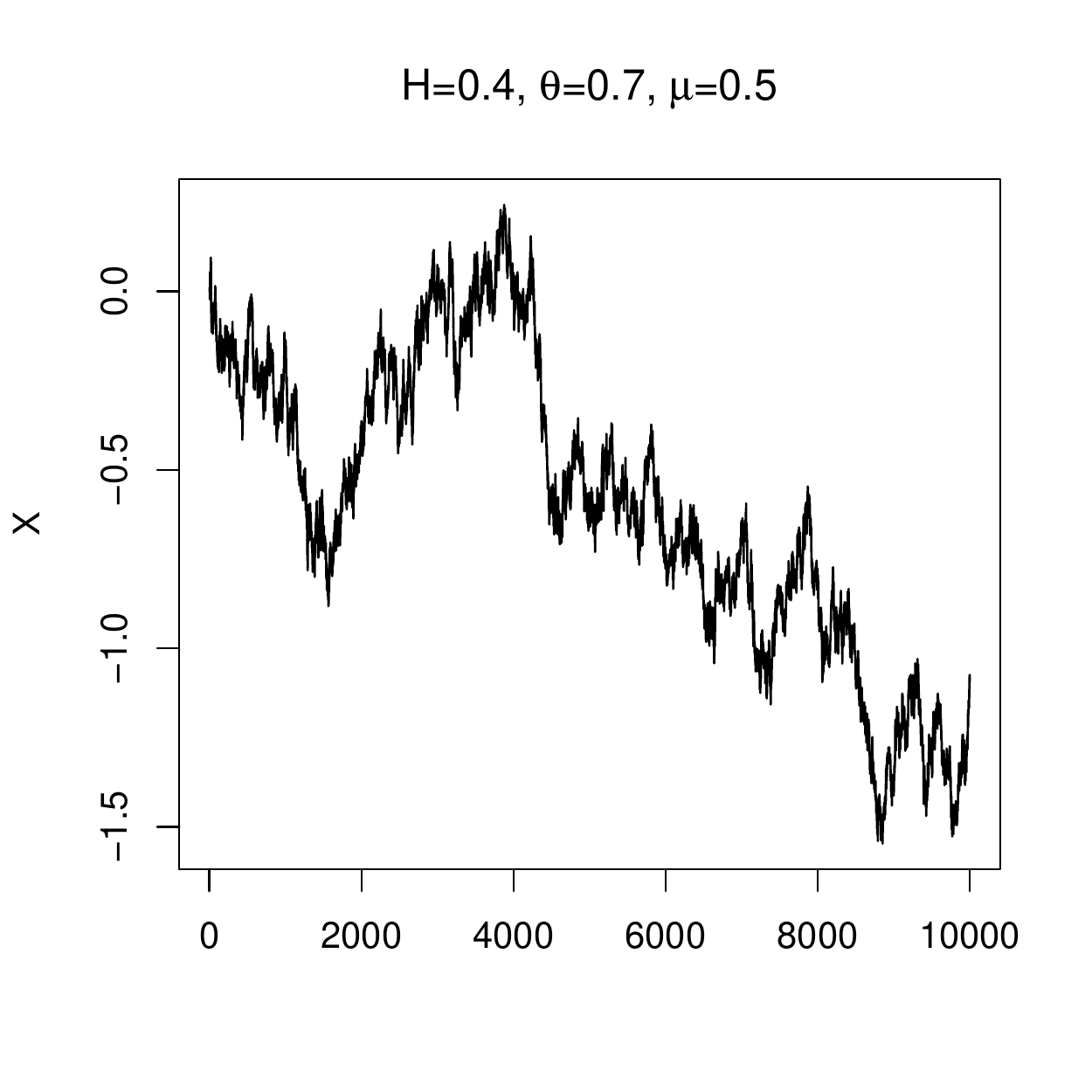}}
        {\includegraphics[ width=3.5cm]{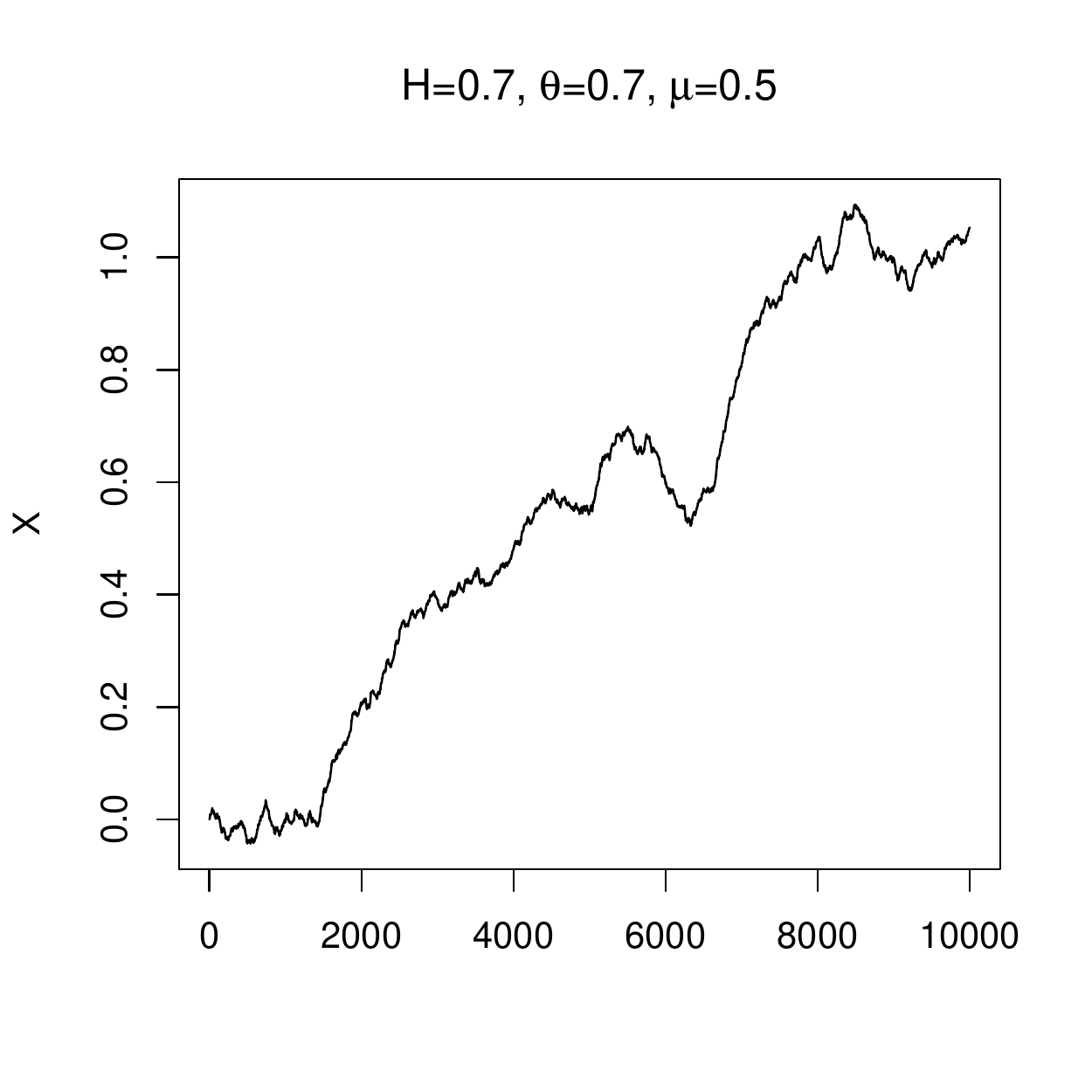}}
        {\includegraphics[ width=3.5cm]{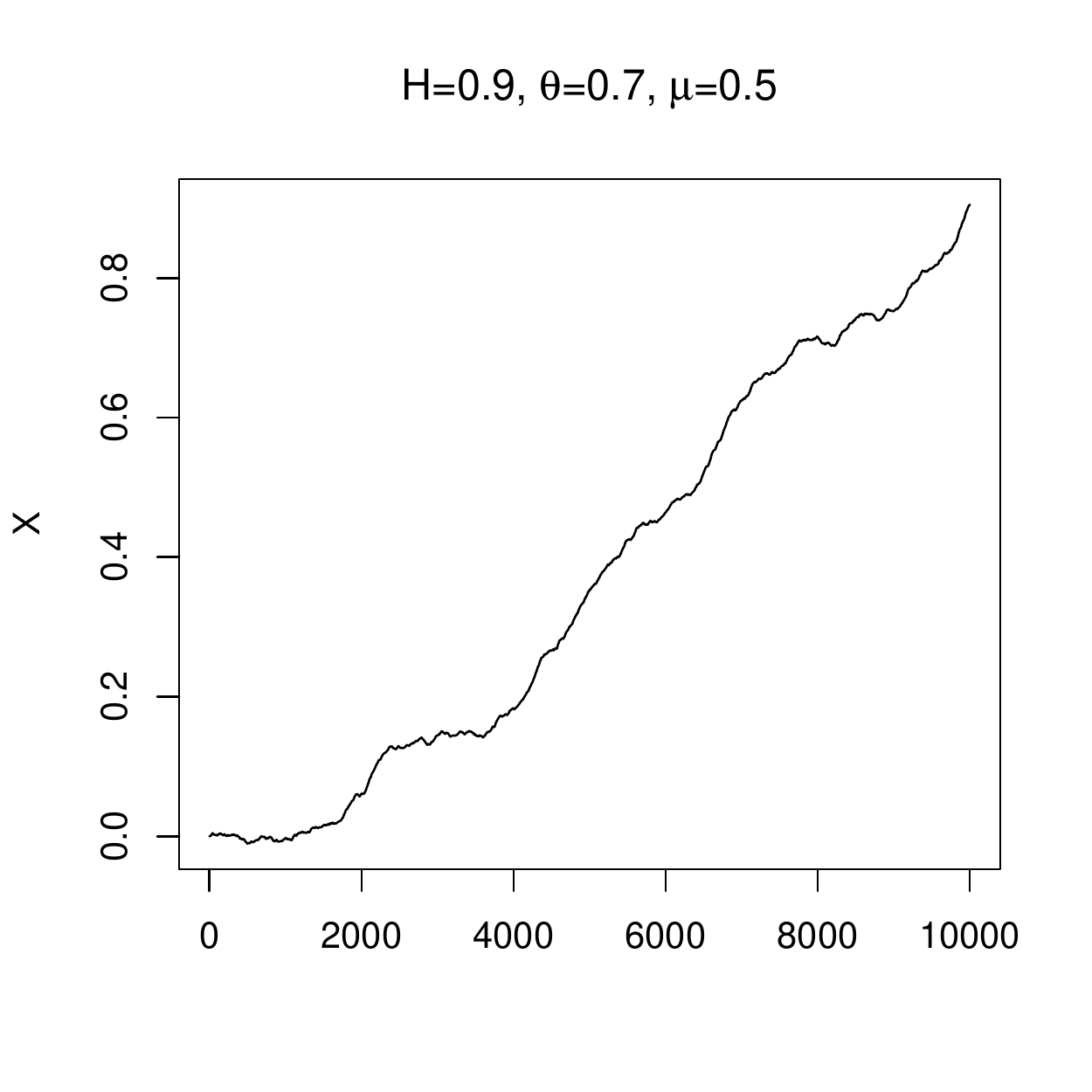}}
    \end{center}
\end{figure}

\noindent From the  properties of $S^H$  listed above, it is clear
that the process $S^H$ satisfies the assumptions $(\mathcal{A}_1)$
and $(\mathcal{A}_2)$ for $G=S^H$ and
 $\gamma=\eta=H$.\\
On the other hand, from \cite[Proposition 3.2.]{EEO}, the
assumptions $(\mathcal{A}_3)$ and $(\mathcal{A}_4)$ hold for
\begin{eqnarray}\sigma_{S^H}^2=\frac{H\Gamma(2H)}{\theta^{2H}},\quad
\lambda_{S^H}^2= 2-2^{2H-1}.\label{sigma-subfBm}\end{eqnarray}
 It remains to check the assumption $(\mathcal{A}_5)$ for $G=S^H$ and $\eta=H$. Fix
 $s\geq0$, using Taylor's expansion
 we have as $T\rightarrow\infty$,
\begin{eqnarray*}\frac{E\left(S^H_sS^H_T\right)}{T^{H}}
&=&\frac{1}{T^{H}}\left(s^{2H}+T^{2H}-\frac12\left[(T+s)^{2H}+(T-s)^{2H}\right]\right)\\
&=&\frac{s^{2H}}{T^{H}}+T^{H}\left(1-\frac12\left[(1+\frac{s}{T})^{2H}+(1-\frac{s}{T})^{2H}\right]\right)
\\
&=&\frac{s^{2H}}{T^{H}}- T^{H} o(1/T)
\\
&=&\frac{s^{2H}}{T^{H}}- T^{H-1} o(1) \longrightarrow0.
\end{eqnarray*}
Furthermore, by   (\ref{IBP}) and the usual integration by parts
formula, we obtain
\begin{eqnarray*}&&E\left[\frac{S^H_T}{T^{H}}
e^{-\theta T}\int_0^Te^{\theta
r}dS^H_r\right]\\&&=E\left[\frac{S^H_T}{T^{H}} \left(S^H_T-\theta
e^{-\theta T} \int_0^Te^{\theta
t}S^H_tdt\right)\right]\\&&=\frac{1}{T^{H}}\left[ (2-
2^{2H-1})T^{2H}- \theta e^{-\theta T} \int_0^Te^{\theta
t}\left(T^{2H}+t^{2H}-\frac12\left[(T+t)^{2H}+(T-t)^{2H}\right]\right)dt\right]
\\&&=\frac{1}{T^{H}}\left[ 2H e^{-\theta T} \int_0^Te^{\theta
t}\left(t^{2H-1}-\frac12\left[(T+t)^{2H-1}-(T-t)^{2H-1}\right]\right)dt\right]\\&&\longrightarrow0
\mbox{ as $T\rightarrow\infty$},
\end{eqnarray*} where the latter convergence comes from the same convergences as in the case of the fBm, and the fact that as
$T\rightarrow\infty$,
\begin{eqnarray*}\frac{e^{-\theta T}}{T^{H}}\int_0^Te^{\theta
t}(T+t)^{2H-1}dt&=&\frac{e^{-2\theta T}}{T^{H}}\int_T^{2T}e^{\theta
x}x^{2H-1}dx\\&\leq& \frac{e^{-2\theta
T}}{T^{H}}(T^{2H-1}+(2T)^{2H-1})\int_T^{2T}e^{\theta
x}dx\longrightarrow0,
\end{eqnarray*}
which proves that   $(\mathcal{A}_5)$ holds.\\
Further, by Lemma \ref{key-applications}, we have for every
$H\in(0,1)$, the variance of $\zeta_{S^H,\infty}:=\theta
\int_0^{\infty}e^{-\theta s}S^H_sds$, given in (\ref{zeta-cv}) when
$G=S^H$, is equal to
\begin{eqnarray}
E(\zeta_{S^H,\infty}^2)&=&\theta^2
\int_0^{\infty}\int_0^{\infty}e^{-\theta s}e^{-\theta
t}E(S^H_sS^H_t)dst\nonumber\\
&=&\theta^2 \left(2k_H-\frac12( l_H+m_H)\right)\nonumber
\\
&=&\frac{1}{2\theta^{2H}} \left(3\Gamma(2H+1)-\Gamma(2H+2)\right)\nonumber\\
&=&\frac{(1-H)\Gamma(2H+1)}{\theta^{2H}}.\label{zeta-subfBm}
\end{eqnarray}
We therefore obtain the following result.
\begin{proposition}
Assume that $H\in(0,1)$ and the process $G$, given in (\ref{GV}), is
a subfBm $S^H$. Then, almost surely, as $T\rightarrow\infty$,
\begin{eqnarray*} \left(\widetilde{\theta}_{T},\widetilde{\mu}_{T}\right)\longrightarrow \left(\theta,\mu\right),
\quad
\left(\widetilde{\theta}_{T},\widetilde{\alpha}_{T}\right)\longrightarrow
\left(\theta,\alpha\right).
 \end{eqnarray*}
 In addition,  if
$N_1\sim\mathcal{N}(0,1)$, $N_2\sim\mathcal{N}(0,1)$ and $S^H$ are
independent, then as $T\rightarrow\infty$,
\begin{eqnarray*}\left(e^{\theta T}(\widetilde{\theta}_T-\theta),T^{1-H}\left(\widetilde{\mu}_T-\mu
\right)\right)\overset{Law}{\longrightarrow}\left(\frac{2\theta\sigma_{S^H}N_2}{\mu
+\zeta_{S^H,\infty}},\frac{\lambda_{S^H}}{\theta}N_1\right),
\end{eqnarray*}
\begin{eqnarray*}\left(e^{\theta
T}(\widetilde{\theta}_T-\theta),T^{1-H}\left(\widetilde{\alpha}_T-\alpha
\right)\right)\overset{Law}{\longrightarrow}\left(\frac{2\theta\sigma_{S^H}N_2}{\mu
+\zeta_{S^H,\infty}},\lambda_{S^H} N_1\right),
\end{eqnarray*}
where $\sigma_{S^H}$ and $\lambda_{S^H}$ are defined in
(\ref{sigma-subfBm}), and $\zeta_{S^H,\infty}\sim
\mathcal{N}(0,E(\zeta_{S^H,\infty}^2))$ is independent of $N_1$ and
$N_2$, with $E(\zeta_{S^H,\infty}^2)$ is given in
(\ref{zeta-subfBm}).
\end{proposition}

\subsection{Bifractional  Vasicek process}\label{section-bifBmV}
Let $B^{H,K} := \left\{ B^{H,K}_t, t\geq 0 \right\}$ be a
bifractional Brownian motion (bifBm)  with parameters $H\in (0, 1)$
and $K\in(0,1]$. This means that $B^{H,K}$ is a centered Gaussian
process with the covariance function
\begin{eqnarray*}
E(B^{H,K}_sB^{H,K}_t)=\frac{1}{2^K}\left(\left(t^{2H}+s^{2H}\right)^K-|t-s|^{2HK}\right);
\quad s,t\geq0.
\end{eqnarray*}
In particular, for every $T\geq0$,
\[\frac{E[(B^{H,K}_T)^2]}{T^{2HK}}=1.\]
 Note that the case $K = 1$ corresponds to the  fBm  with Hurst parameter $H$.\\ In addition, the process $B^{H,K}$ verifies, \begin{eqnarray*}
 E\left(\left|B^{H,K}_t-B^{H,K}_s\right|^2\right)\leq
2^{1-K}|t-s|^{2HK}.
\end{eqnarray*}
Let us first start with the following simulated path of the
subfractional Vasicek process, i.e., when $G=B^{H,K}$ in (\ref{GV}),
\begin{eqnarray}X_0=0;\quad dX_t=\theta\left(\mu+X_t\right)dt+dB^{H,K}_t.
\label{bifBmV}\end{eqnarray}
\begin{itemize}
    \item First, we generate the bifractional Brownian motion using the \textbf{R} package \textbf{FieldSim} (see \cite{BIL}).
    \item After that we simulate the process (\ref{bifBmV}) using the Euler-Maruyama method for
    different values of $H$, $\theta$ and $\mu$ (see Figure \ref{bifvp}).
\end{itemize}
We simulate a sample path on the interval $[0,1]$ using a regular
partition of 10,000 intervals.

\begin{figure}[H]
    \caption{The sample path of bifractional Vasicek process. }
    \label{bifvp}
    \begin{center}
        {\includegraphics[ width=3.5cm]{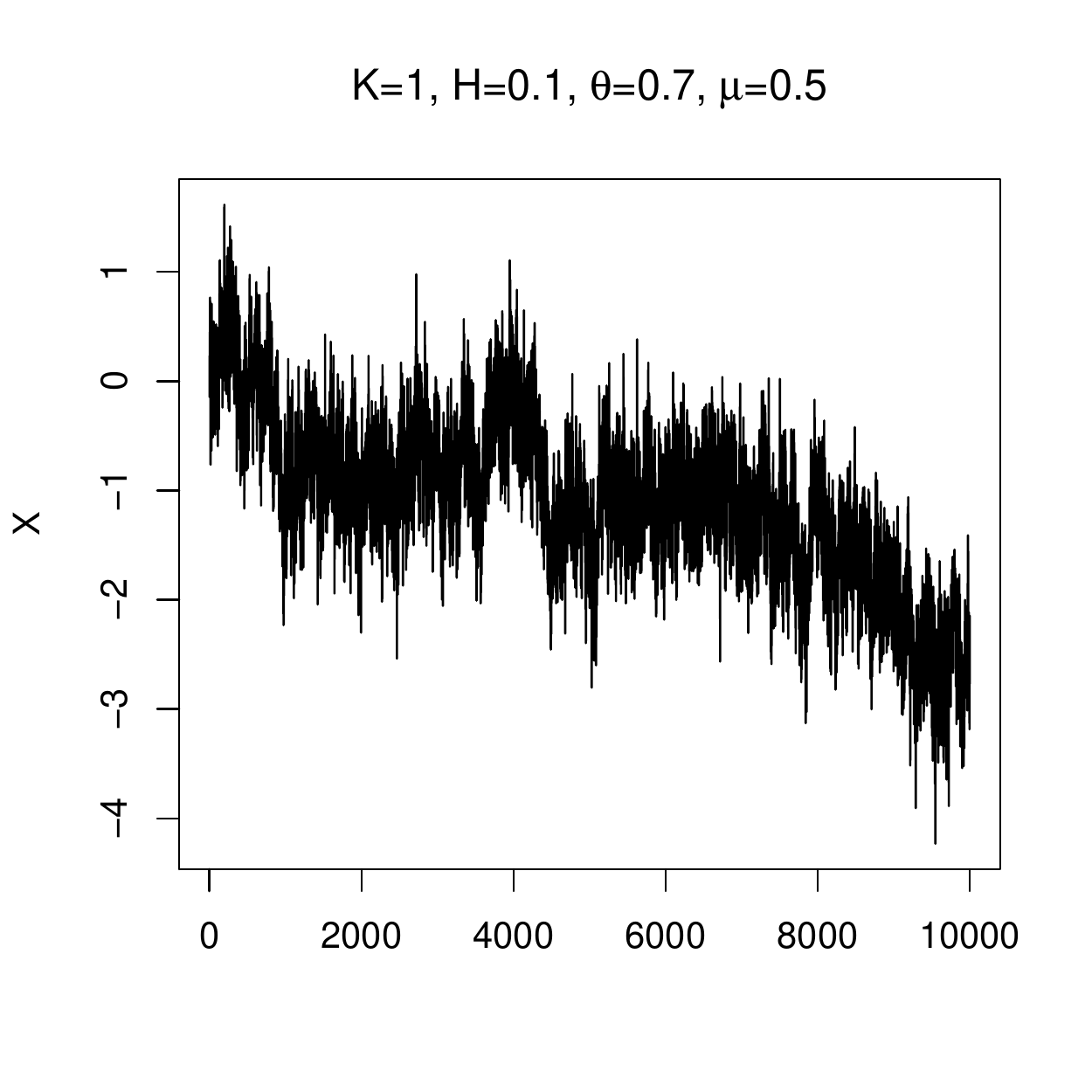}}
        {\includegraphics[ width=3.5cm]{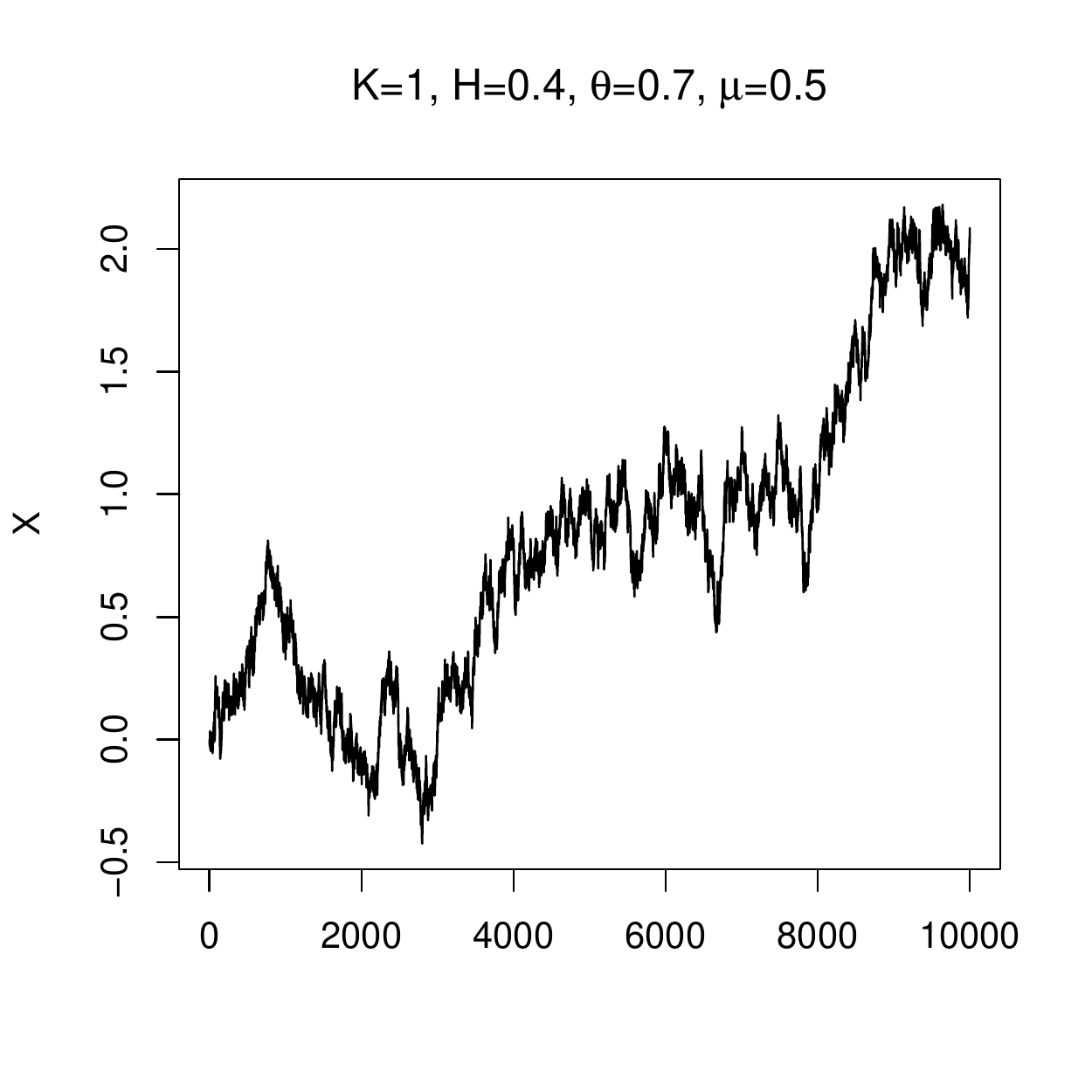}}
        {\includegraphics[ width=3.5cm]{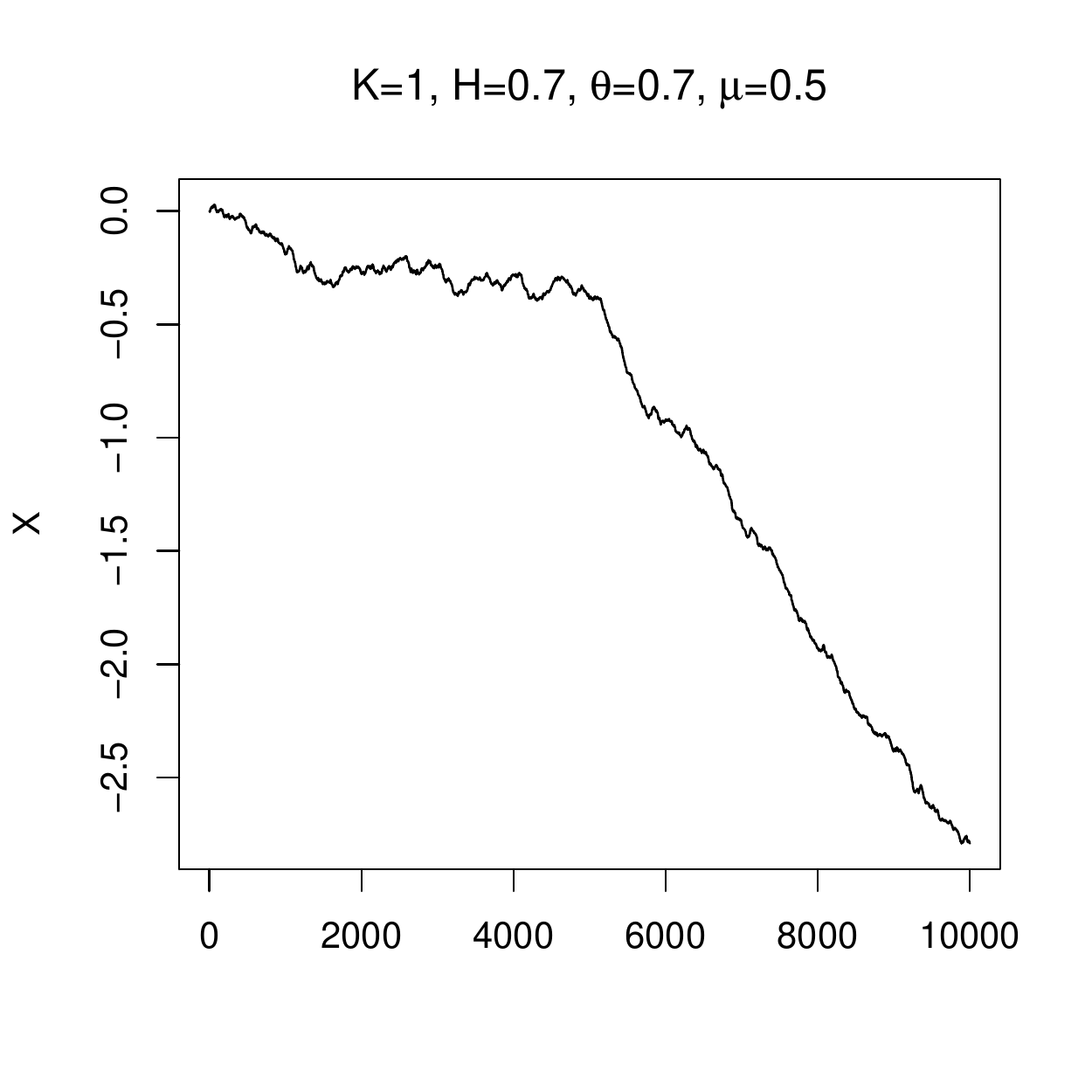}}
        {\includegraphics[ width=3.5cm]{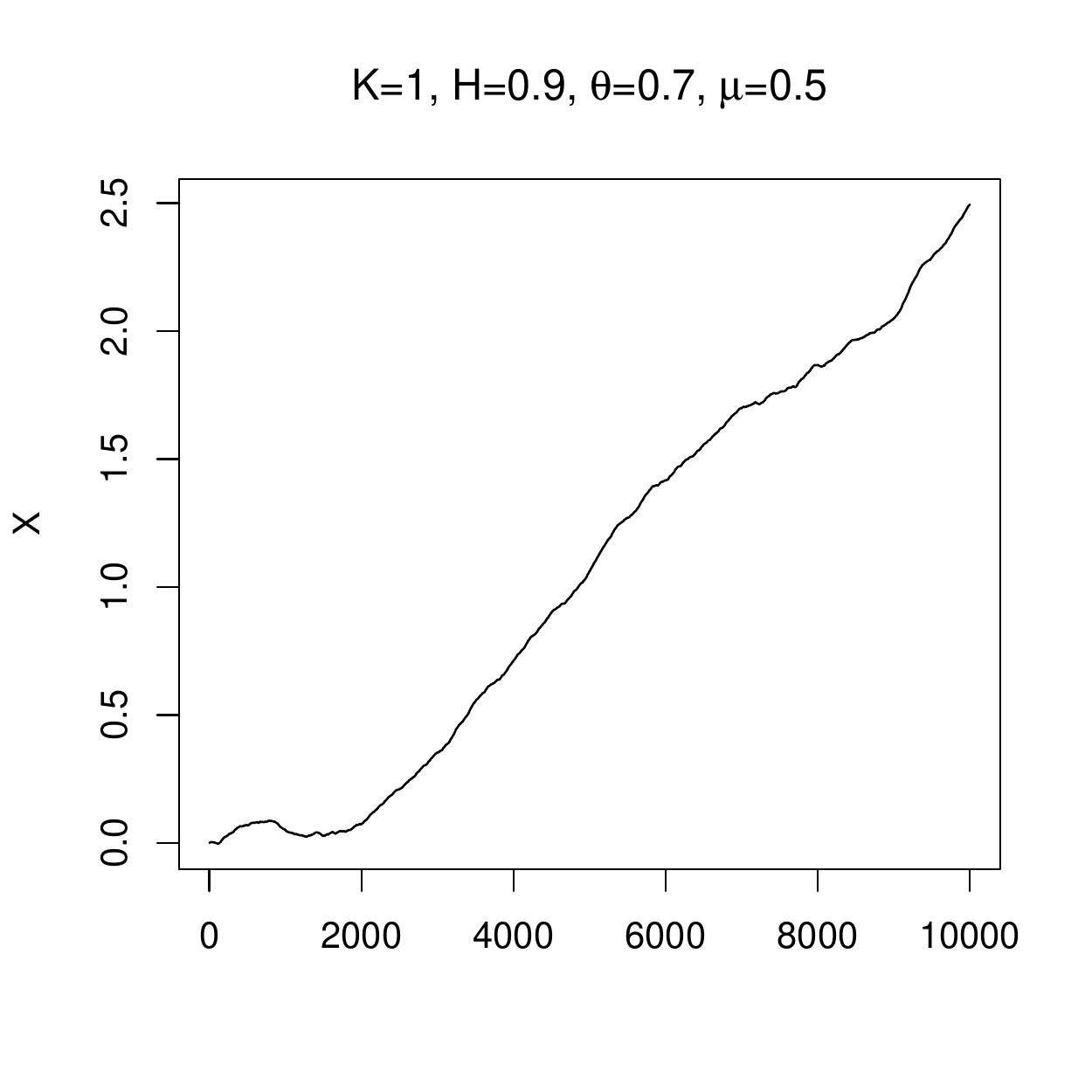}}
        {\includegraphics[ width=3.5cm]{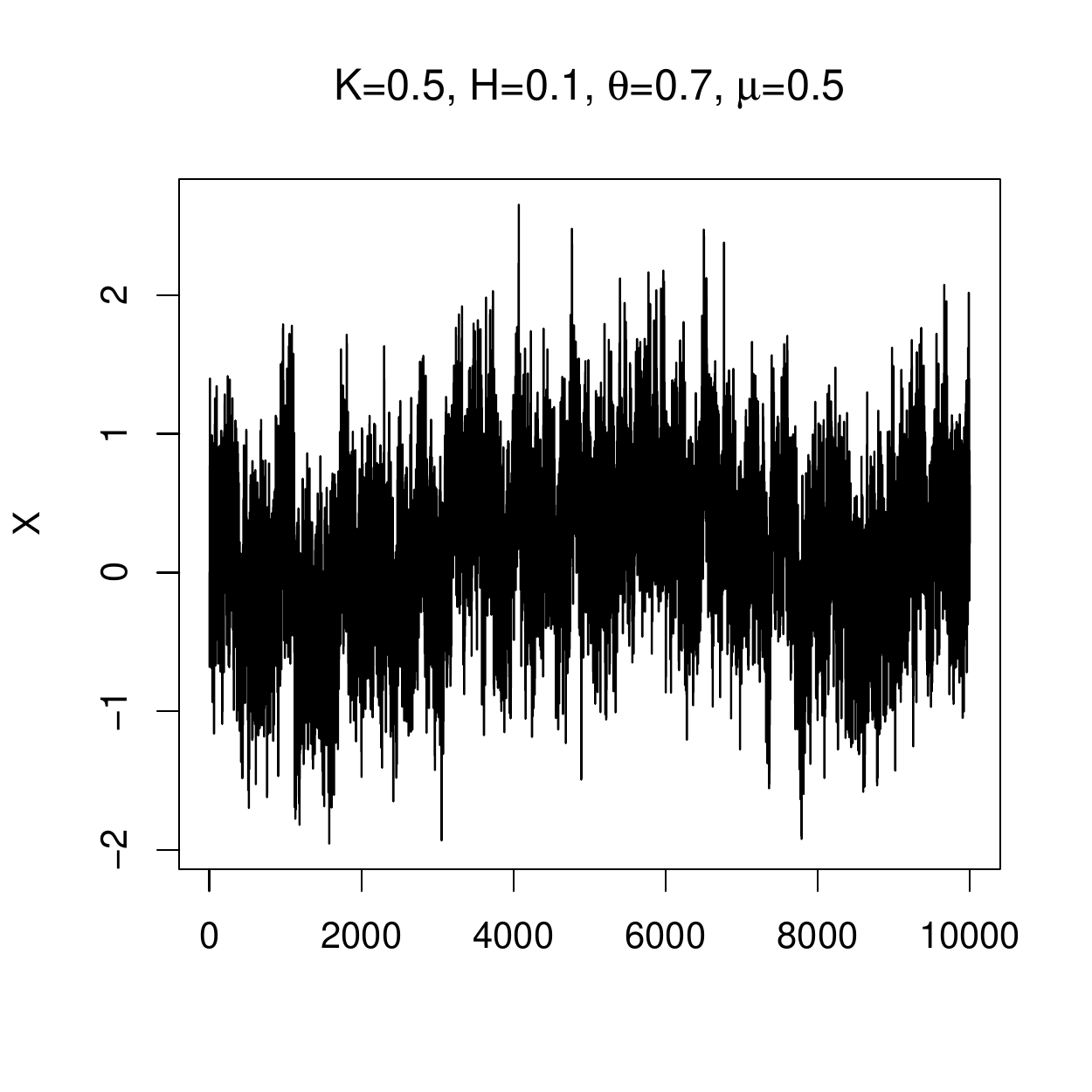}}
        {\includegraphics[ width=3.5cm]{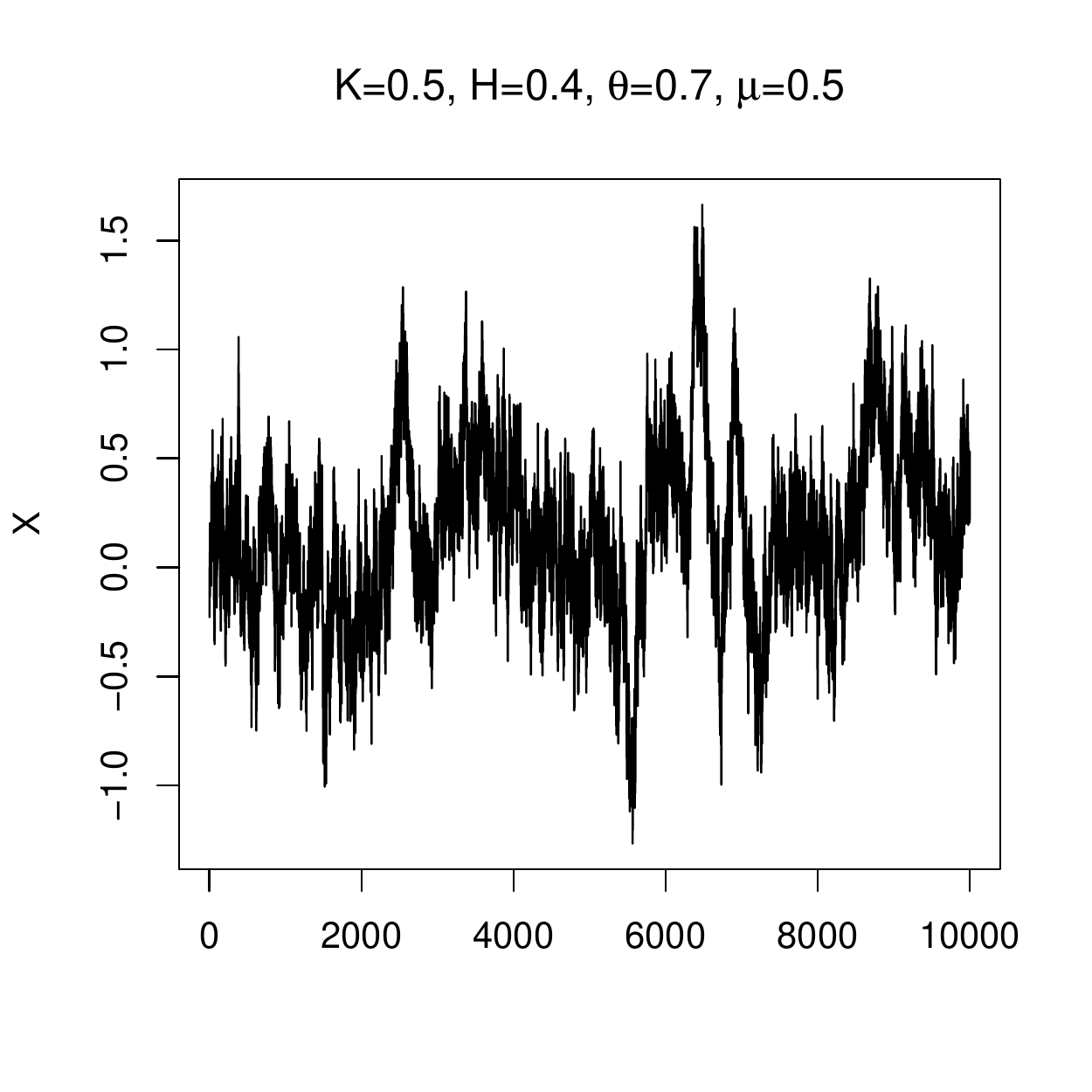}}
        {\includegraphics[ width=3.5cm]{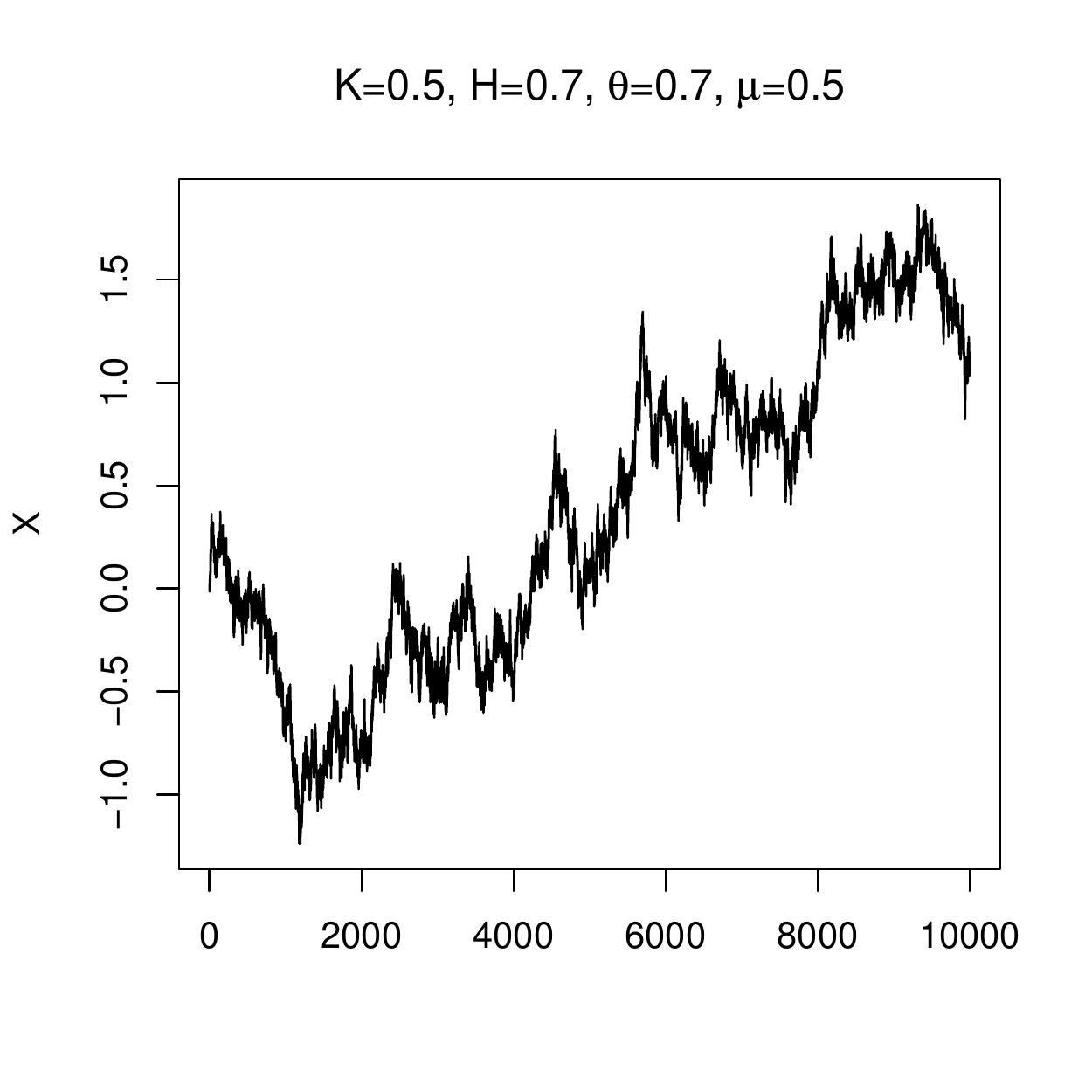}}
        {\includegraphics[ width=3.5cm]{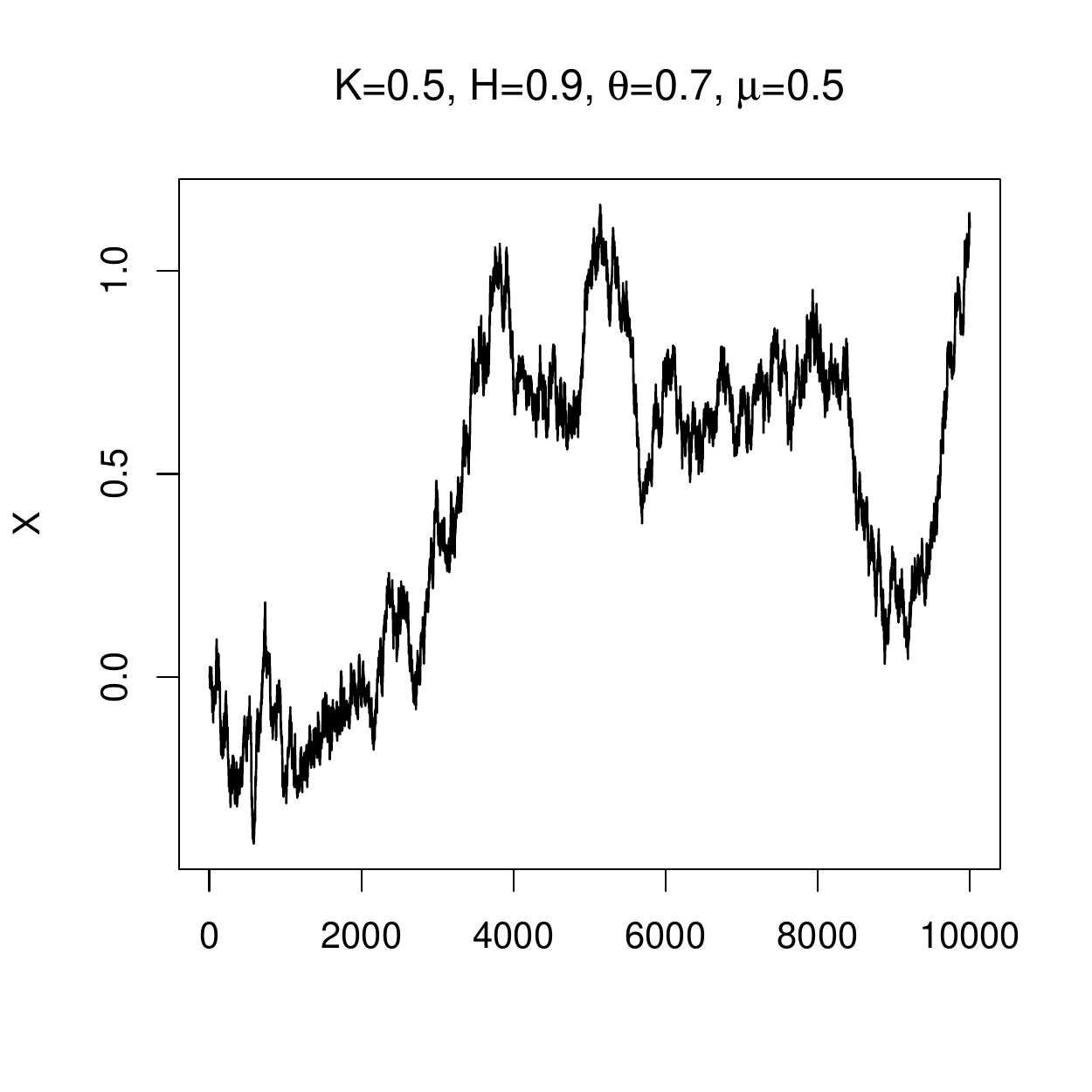}}
    \end{center}
\end{figure}

\noindent From the  properties of  $B^H$  listed above, it is clear
that the process $B^H$ satisfies the assumptions $(\mathcal{A}_1)$
and $(\mathcal{A}_2)$ for $G=B^{H,K}$ and
 $\gamma=\eta=HK$.\\
Furthermore, by \cite[Proposition 3.3.]{EEO}, the assumptions
$(\mathcal{A}_3)$ and $(\mathcal{A}_4)$  hold for
\begin{eqnarray}\sigma_{B^{H,K}}^2=\frac{HK\Gamma(2HK)}{\theta^{2HK}},\quad
\lambda_{B^{H,K}}= 1.\label{sigma-bifBm}\end{eqnarray} It remains to
check the assumption $(\mathcal{A}_5)$ for $G=B^{H,K}$ and
$\eta={HK}$. Fix
 $s\geq0$, applying Taylor's expansion to the functions $(1+x^{2H})^K$
 and $(1-x)^{2HK}$, we obtain as $T\rightarrow\infty$,
\begin{eqnarray*}\frac{E\left(B^{H,K}_sB^{H,K}_T\right)}{T^{{HK}}}
&=&\frac{1}{2^KT^{HK}}\left[(s^{2H}+T^{2H})^K-(T-s)^{2HK}\right]\\
&=&\frac{T^{HK}}{2^K}\left[(1+(\frac{s}{T})^{2H})^K-(1-\frac{s}{T})^{2HK}\right]\\
&=&\frac{T^{HK}}{2^K}\left[K(s/T)^{2H}+o(1/T^{2H})+2HKs/T+o(1/T)\right]\\
&=&\frac{1}{2^K}\left[K(s^{2H}/T^{H(2-K)})+o(1/T^{H(2-K)})+2HKs/T^{1-HK}+o(1/T^{1-HK})\right]\\&&
\longrightarrow0.
\end{eqnarray*}
Furthermore, by   (\ref{IBP}) and the usual integration by parts
formula, we obtain
\begin{eqnarray*}&&E\left[\frac{B^{H,K}_T}{T^{{HK}}}
e^{-\theta T}\int_0^Te^{\theta
r}dB^{H,K}_r\right]\\&&=E\left[\frac{B^{H,K}_T}{T^{{HK}}}
\left(B^{H,K}_T-\theta e^{-\theta T} \int_0^Te^{\theta
t}B^{H,K}_tdt\right)\right]\\&&=\frac{1}{T^{{HK}}}\left[
T^{2{HK}}-\frac{\theta e^{-\theta T}}{2^K}  \int_0^Te^{\theta
t}\left((T^{2{H}}+t^{2{H}})^K-(T-t)^{2{HK}}\right)dt\right]\\&&=
\frac{1}{T^{HK}}\left[ T^{2{HK}}e^{-\theta T}{2^K}-\frac{2HK
e^{-\theta T}}{2^K} \int_0^Te^{\theta
t}t^{2H-1}\left(T^{2H}+t^{2H}\right)^{K-1}dt\right]\\&&\quad+\frac{\theta
e^{-\theta T}}{2^KT^{HK}}  \int_0^Te^{\theta t}(T-t)^{2{HK}}dt,
\end{eqnarray*} which converges to zero as $T\rightarrow\infty$, where we used
that $\frac{ e^{-\theta T}}{T^{HK}}  \int_0^Te^{\theta
t}(T-t)^{2{HK}}dt$, by the same arguments as in the case of fBm,
 and using  L'H\^{o}pital's rule,
\begin{eqnarray*} \lim_{T\rightarrow\infty}
\frac{e^{-\theta T}}{T^{HK}} \int_0^Te^{\theta
t}t^{2H-1}\left(T^{2H}+t^{2H}\right)^{K-1}dt&\leq&
\lim_{T\rightarrow\infty} \frac{e^{-\theta T}}{T^{HK}}
\int_0^Te^{\theta t}t^{2HK-1}dt\\
&=&\lim_{T\rightarrow\infty}\frac{T^{2HK-1}}{\theta
T^{HK}+HKT^{HK-1}}
\\&=&\lim_{T\rightarrow\infty}\frac{T^{HK-1}}{\theta +\frac{HK}{T}}\\&=&0.
\end{eqnarray*}
Thus the assumption $(\mathcal{A}_5)$ holds.  As consequence, we
obtain the following result.
\begin{proposition}
Assume that $(H,K)\in(0,1)\times(0,1]$ and the process $G$, given in
(\ref{GV}), is a bifBm $B^{H,K}$. Then, almost surely, as
$T\rightarrow\infty$,
\begin{eqnarray*} \left(\widetilde{\theta}_{T},\widetilde{\mu}_{T}\right)\longrightarrow \left(\theta,\mu\right),
\quad
\left(\widetilde{\theta}_{T},\widetilde{\alpha}_{T}\right)\longrightarrow
\left(\theta,\alpha\right).
 \end{eqnarray*}
 In addition, if
$N_1\sim\mathcal{N}(0,1)$, $N_2\sim\mathcal{N}(0,1)$ and $B^{H,K}$
are independent, then as  $T\rightarrow\infty$,
\begin{eqnarray*}\left(e^{\theta T}(\widetilde{\theta}_T-\theta),T^{1-HK}\left(\widetilde{\mu}_T-\mu
\right)\right)\overset{Law}{\longrightarrow}\left(\frac{2\theta\sigma_{B^{H,K}}N_2}{\mu
+\zeta_{B^{H,K},\infty}},\frac{1}{\theta}N_1\right),
\end{eqnarray*}
\begin{eqnarray*}\left(e^{\theta T}(\widetilde{\theta}_T-\theta),T^{1-HK}\left(\widetilde{\alpha}_T-\alpha
\right)\right)\overset{Law}{\longrightarrow}\left(\frac{2\theta\sigma_{B^{H,K}}N_2}{\mu
+\zeta_{B^{H,K},\infty}}, N_1\right),
\end{eqnarray*}
where $\sigma_{B^{H,K}}$ is defined in (\ref{sigma-bifBm}), and
$\zeta_{B^{H,K},\infty}=\theta \int_0^{\infty}e^{-\theta
s}B^{H,K}_sds\sim \mathcal{N}(0,E(\zeta_{B^{H,K},\infty}^2))$ is
independent of $N_1$ and $N_2$, with
$E(\zeta_{B^{H,K},\infty}^2)<\infty$, by Lemma
\ref{key-applications}.
\end{proposition}

\section{Appendix: Young integral}
In this section, we briefly recall some basic elements of Young
integral (see \cite{Young}),  which are helpful for some of the
arguments we use. For any $\alpha\in [0,1]$, we denote by
$\mathscr{H}^\alpha([0,T])$ the set of $\alpha$-H\"older continuous
functions, that is, the set of functions $f:[0,T]\to\R$ such that
\[
|f|_\alpha := \sup_{0\leq s<t\leq
T}\frac{|f(t)-f(s)|}{(t-s)^{\alpha}}<\infty.
\]
We also set $|f|_\infty=\sup_{t\in[0,T]}|f(t)|$, and we equip
$\mathscr{H}^\alpha([0,T])$ with the norm \[ \|f\|_\alpha :=
|f|_\alpha + |f|_\infty.\] Let $f\in\mathscr{H}^\alpha([0,T])$, and
consider the operator $T_f:\mathcal{C}^1([0,T])
\to\mathcal{C}^0([0,T])$ defined as
\[
T_f(g)(t)=\int_0^t f(u)g'(u)du, \quad t\in[0,T].
\]
It can be shown (see, e.g., \cite[Section 3.1]{nourdin}) that, for
any $\beta\in(1-\alpha,1)$, there exists a constant
$C_{\alpha,\beta,T}>0$ depending only on $\alpha$, $\beta$ and $T$
such that, for any $g\in\mathcal{C}^1([0,T])$,
\[
\left\|\int_0^\cdot f(u)g'(u)du\right\|_\beta \leq
C_{\alpha,\beta,T} \|f\|_\alpha \|g\|_\beta.
\]
We deduce that, for any $\alpha\in (0,1)$, any
$f\in\mathscr{H}^\alpha([0,T])$ and any $\beta\in(1-\alpha,1)$, the
linear operator
$T_f:\mathcal{C}^1([0,T])\subset\mathscr{H}^\beta([0,T])\to
\mathscr{H}^\beta([0,T])$, defined as $T_f(g)=\int_0^\cdot
f(u)g'(u)du$, is continuous with respect to the norm
$\|\cdot\|_\beta$. By density, it extends (in an unique way) to an
operator defined on $\mathscr{H}^\beta$. As consequence, if
$f\in\mathscr{H}^\alpha([0,T])$, if $g\in\mathscr{H}^\beta([0,T])$
and if $\alpha+\beta>1$, then the (so-called) Young integral
$\int_0^\cdot f(u)dg(u)$ is (well) defined as being $T_f(g)$.

The Young integral obeys the following formula. Let
$f\in\mathscr{H}^\alpha([0,T])$ with $\alpha\in(0,1)$, and
$g\in\mathscr{H}^\beta([0,T])$ with $\beta\in(0,1)$. If
$\alpha+\beta>1$, then $\int_0^. g_udf_u$ and $\int_0^. f_u dg_u$
are well-defined as Young integrals, and for all $t\in[0,T]$,
\begin{eqnarray}\label{IBP}
f_tg_t=f_0g_0+\int_0^t g_udf_u+\int_0^t f_u dg_u.
\end{eqnarray}

\end{document}